\documentclass[11pt,reqno]{amsart}

\usepackage[varg]{txfonts}

\usepackage{anyfontsize}

\usepackage{amssymb, amsmath, amsthm, amsbsy, mathrsfs, mathtools, mathdots, bbm, stackrel, cancel, ulem, graphicx, relsize, dynkin-diagrams, upgreek} 

\usepackage[toc, page]{appendix}

\usepackage[all, cmtip]{xy}
\usepackage{tikz, tikz-cd}
\usetikzlibrary{math}

\usepackage[mathscr]{euscript}
\usepackage[scr=boondoxo]{mathalfa}

\usepackage{setspace}
\usepackage[shortlabels]{enumitem}
\setlist[enumerate]{itemsep=1pt, topsep=2pt}
\setlist[itemize]{itemsep=1pt, topsep=2pt}

\usepackage{color, hyperref}
\hypersetup{colorlinks=true, linktoc=all, linkcolor=blue, citecolor=blue}

\definecolor{mygreen}{RGB}{0,170,0}

\DeclareMathOperator{\soc}{Soc}

\DeclareMathOperator{\en}{End}

\DeclareMathOperator{\id}{Id}
\makeatletter
\newcommand{\oset}[3][0ex]{%
  \mathrel{\mathop{#3}\limits^{
    \vbox to#1{\kern-2\ex@
    \hbox{$\scriptstyle#2$}\vss}}}}
\makeatother

\newcommand{\tl}[1]{\tilde{#1}}
\newcommand{\wtl}[1]{\widetilde{#1}}

\newcommand{\A}{\alpha}
\newcommand{\B}{\beta}
\newcommand{\G}{\gamma}
\newcommand{\D}{\delta}
\newcommand{\la}{\lambda}

\newcommand{\C}{\mathbb{C}}

\newcommand{\Z}{\mathbb{Z}}

\newcommand{\tb}[1]{\textbf{#1}}

\newcommand{\ti}[1]{\textit{#1}}

\newcommand{\ol}[1]{\overline{#1}}

\newcommand{\mr}[1]{\mathrm{#1}}

\newcommand{\fk}[1]{\mathfrak{#1}}

\newcommand{\p}[1]{\begin{proof}#1\end{proof}}

\newcommand{\dia}[1]{\vspace{-2mm}\begin{center}\begin{tikzcd}[ampersand replacement = \&]#1\end{tikzcd}\end{center}\vspace{-1mm}}
\newcommand{\ddia}[2]{\vspace{-2mm}\begin{center}\begin{tikzcd}[ampersand replacement = \&, #1]#2\end{tikzcd}\end{center}\vspace{-1mm}}
\newcommand{\eq}[1]{\begin{equation}\begin{aligned}#1\end{aligned}\end{equation}}
\numberwithin{equation}{section}

\newcommand{\comment}[1]{}

\newcommand{\bee}[1]{$$\begin{aligned}#1\end{aligned}$$}

\theoremstyle{plain}
\newtheorem{thm}{Theorem}[section]

\newtheorem{prp}[thm]{Proposition}
\newtheorem{cor}[thm]{Corollary}

\theoremstyle{definition}
\newtheorem{dfn}[thm]{Definition}

\theoremstyle{remark}
\newtheorem{rmk}[thm]{Remark}

\author{}
\date{\today}
\title{Intertwiners of representations of exceptional type quantum affine superalgebras}

\begin{document}

\author{Keshav Dahiya and Evgeny Mukhin} 
\address{EM: Department of Mathematical Sciences,
Indiana University Indianapolis,
402 N. Blackford St., LD 270, 
Indianapolis, IN 46202, USA}
\email{emukhin@iu.edu} 

\address{KD: Department of Mathematical Sciences,
Indiana University Indianapolis,
402 N. Blackford St., LD 270, 
Indianapolis, IN 46202, USA}
\email{kkeshav@iu.edu} 

\begin{abstract}
    We give explicit formulas for the  smallest non-trivial irreducible representation $V$ of quantum affine superalgebras in types D$_{2\vert 1;\alpha}$, $\dim V=18$ (in the all-fermionic parity), F$_{3\vert 1}$, $\dim V=41$ (in the distinguished parity), and G$_{2\vert 1}$, $\dim V=32$ (in the distinguished parity), both in the Drinfeld-Jimbo and in the new Drinfeld realizations. We use this information to obtain an explicit expression for the corresponding $R$-matrices.
    
\medskip

\centerline{
  }

  \centerline{
  \textbf{\textit{AMS Classification numbers:\ }}{81R10 (primary), 16T25, 17B10, 17B37.}}
  
\end{abstract}

\maketitle

\section{Introduction} 
Solutions of the quantum Yang-Baxter equation (QYBE) are important for various reasons. They have applications in the exactly solvable integrable systems, quantum groups, knot invariants, and quantum computing.

This is the third in a series of papers on finding explicit expressions for solutions of  QYBE $\check{R}(z)$ which are intertwiners of finite-dimensional representations of quantum affine algebras $U_q\tl{\mathfrak{g}}$. In \cite{DM25a} we developed a method for obtaining such solutions using the theory of $q$-characters and applied it to the case of $V(z)\otimes V$, where $V$ is the first fundamental representation, for all untwisted quantum affine algebras. In \cite{DM25b} we treated the tensor squares of the first fundamental representations of all twisted type quantum affine algebras. In the current text, we consider the basic super-symmetric cases. The $R$-matrices for first fundamental representations for $\mathfrak{g}$ of types $\mathfrak{sl}_{m|n}$, $\mathfrak{osp}_{m|2n}$ are known, see Sections \ref{super:sl} and \ref{super:osp}, and we treat all the remaining cases. We give three new solutions of the QYBE, related to the smallest non-trivial irreducible representations $V$ of the three basic exceptional type quantum affine superalgebras $U_q\tl{\mathfrak{g}}$, where $\mathfrak{g}$ is of types D$_{2\vert 1;\alpha}$ in the all-fermionic parity, F$_{3\vert 1}$ in the distinguished parity, and G$_{2\vert 1}$ in the distinguished parity. In all three cases,  $V$, as $U_q\mathfrak{g}$-module, is a quantization of $\mathfrak g\oplus \C$. In particular,
the corresponding dimensions of $V$ are $18$, $41$, and $32$, respectively.

The answer, see Theorems \ref{superD:Rcheck}, \ref{superF:Rcheck}, \ref{superG:Rcheck}, is given in terms of maps of $U_q\mathfrak{g}$-modules. Due to nontrivial multiplicities, the essential part of the answer is given by $2\times 2$  and $4\times 4$ matrices in the case of D$_{2\vert 1;\alpha}$, and by $2\times 2$ and $3\times 3$ matrices in the cases of   F$_{3\vert 1}$, and G$_{2\vert 1}$. While the structure of these matrices is similar to the ones appearing in the even cases, we observe a number of new features. In the cases of  F$_{3\vert 1}$ and  G$_{2\vert 1}$, for some values of the spectral parameter $z$, the composition series of $V(z)\otimes V$ consists of three irreducible subfactors and all of them can be found from the $R$-matrix. In these cases, $R$-matrix has a pole both at $z$ and $z^{-1}$. In the case of D$_{2\vert 1;\alpha}$, $V\otimes V$ is not completely reducible as $U_q\mathfrak{g}$-module, and $\check{R}(0)$ has a one-dimensional kernel making $V$ a non-real $U_q\tl{\mathfrak{g}}$-module.

The quantum affine superalgebras are much less understood compared to the even case. Several foundational results such as establishing the Drinfeld loop realization and a classification of finite-dimensional irreducible representations are still open problems for almost all cases. The theory of the $q$-characters is also less developed due to the absence of the important concept of the dominant monomial for fermionic roots.

This provides additional difficulties to applying our method to obtain $R$-matrices. In particular, one challenge is to compute the submodules of the tensor product $V(z)\otimes V$ using the theory of $q$-characters. The analog of Theorem 6.7 in \cite{FM01}, which relates the poles of $R$-matrices to the irreducibility of the tensor product, should still carry forward to the supersymmetric cases, however, a rigorous proof is still to be given. 
The finite type $R$-matrix, which is used for finding $\check{R}(0)$ is also not available.
A method for obtaining $R$-matrices using the extended tensor product graph method for supersymmetric cases can be found in \cite{GZ03}. However, it is applied only for cases with trivial multiplicity.

As a result, we end up writing the explicit formulas for the representations in Drinfeld-Jimbo realization and then obtain the decompositions of tensor products and the intertwiners using computer computations. We also give explicit formulas for the same representations in terms of (conjectured) new Drinfeld realization. In particular, the goal of this paper is to provide some basic but important information on several examples of tensor products rather than proving general theorems. The explicit actions, tensor product decompositions, and even dimensions of $U_q\tl{\mathfrak{g}}$-modules are often unknown. We are not aware of any ways at all which could help to avoid brute force computations and obtain this information systematically.

Note that the comultiplication depends on the choice of Dynkin diagram. The $R$-matrices in different choices must be similar. It would be interesting to find the similarity transformation explicitly.

\medskip 

The paper is organized as follows. In Section \ref{super:QAA}, we put together the known facts about quantum affine superalgebras and discuss the open problems in some more detail. In Sections \ref{super:sl} and \ref{super:osp}, we write the already known $R$-matrices for the first fundamental representations of $U_q\wtl{\mathfrak{sl}}_{m\vert n}$ and $U_q\wtl{\mathfrak{osp}}_{m\vert 2n}$. 
In Sections \ref{super:D}, \ref{super:F} and \ref{super:G}, we give the $R$-matrices for  non-trivial representations of least dimension, in the exceptional types, in the all fermionic parity for the case of D$_{2\vert 1;\A}$, and in the distinguished parity for the cases of F$_{3\vert 1}$, G$_{2\vert 1}$. 

\section{Quantum affine superalgebras}
\label{super:QAA}

We use the following general notation.

\begin{enumerate}

\item Let $\mathfrak{g}$ be the Lie superalgebra of type $\mathfrak{gl}_{m\vert n}$, $\mathfrak{osp}_{m\vert 2n}$, D$_{2\vert 1;\alpha}$, F$_{3\vert 1}$, or G$_{2\vert 1}$. Let $\mr{I}$ be the set of nodes of a Dynkin diagram of $\mathfrak{g}$, and let $C=(C_{ij})_{i,j\in\mr{I}}$ be the corresponding Cartan matrix. 
Let $r=\lvert\mr{I} \rvert$ be the number of nodes in the Dynkin diagram chosen.

\item For $i\in\mr{I}$, let $s_i$ be the parity of the $i$-th node. We set $s_i=0$ if the $i$-th root is even (or bosonic) and $s_i=1$ if the $i$-th root is odd (or fermionic). For all cases we consider in this paper, $s_i=1$ if and only if $C_{ii}=0$. (In other words, we do not consider non-isotropic fermionic nodes.) Let $D=\text{diag}(d_1,\dots,d_r)$ be such that $B=DC$ is symmetric, and $d_i\in\mathbb{Z}$ are relatively prime.

\item Let $\tl{\mathfrak{g}}=\mathfrak{g}\otimes \mathbb{C}[t,t^{-1}]$ be the loop Lie algebra associated to $\mathfrak{g}$. Let $\tl{\mr{I}}=\mr{I}\cup\{0\}$. Let  $\tl{C}=(\tl{C}_{ij})_{i,j\in\tl{\mr{I}}}$ and $\tl{B}=(\tl{B}_{ij})_{i,j\in\tl{\mr{I}}}$ be the corresponding affine Cartan and symmetrized Cartan matrices. Let $a=(a_0,\dots,a_r)$ be the sequence of positive integers such that $\tl Ca^T=0$ and such that $a_0,\dots, a_r$ are relatively prime.

\comment{
\item Let $\A_i$, $i\in\tl{\mr{I}}$, be the simple roots corresponding to a chosen Dynkin diagram. Denote by $(.,.)$ the symmetric bilinear form induced on the set of all roots by $(\A_i,\A_j)=\tl{B}_{ij}$.

\item In a Dynkin diagram, an even root will be denoted by $\bigcirc$, an odd root $\A$ with $(\A,\A)=0$ will be denoted by $\bigotimes$, and an odd root $\A$ with $(\A,\A)\ne0$ will be denoted by $\bullet$. A cross $\times$ will denote $\bigcirc$ or $\bigotimes$.
}



\item Let $q\in\C^\times$ be such that $q$ is not a root of unity. We fix a value of logarithm of $q$ and using that define $q^a$ for any $a\in\C$.
Let $q_j=q^{d_j}$, $j\in \tl{\mr{I}}$. For $k,m\in\C$, such that $q^{2k}\neq \pm 1$, and $n\in\Z$,  set
$$[m]_k=\frac{q^{km}-q^{-km}}{q^k-q^{-k}}\ ,\quad [n]_k^{\mr{i}}=\frac{q^{kn}+(-1)^{n-1}q^{-kn}}{q^k+q^{-k}}.$$
We write $[m]_1$ as $[m]$ and $[n]_1^{\mr{i}}$ as $[n]^{\mr{i}}$.\\
Note that $\lim_{q\to 1}\,[m]_k=m$, $\lim_{q\to 1}[n]_k^{\mr{i}}=1$ if $n$ is an odd integer, and $\lim_{q\to 1}[n]_k^{\mr{i}}=0$ if $n$ is an even integer.

\item All representations are assumed to be finite-dimensional. We consider quantum affine algebras of level zero only.

\item We use the standard convention of supersymmetric spaces. Given operators $A$, $B$ and vectors $v$, $w$, $(A\otimes B)(v\otimes w)=(-1)^{p_B\,p_v}(Av\otimes Bw)$, where $p_B$, $p_v$ are parities of operator $B$ and vector $v$ respectively.
\end{enumerate}

\subsection{Drinfeld-Jimbo realization}

\begin{dfn}\label{DJ:super}
    The quantum affine superalgebra $U_q\tl{\mathfrak{g}}$ (at level zero), corresponding to a choice of Dynkin diagram, is the unital associative algebra over $\mathbb{C}$ with generators $K_i^{\pm 1}$, $E_i$, $F_i$, $i\in \tl{\mr{I}}$, modulo relations: 
    $$K_iK_i^{-1}=K_i^{-1}K_i=1\,\,,\quad K_iK_j=K_jK_i\,\,,\quad K_0^{a_0}K_1^{a_1}\cdots K_r^{a_r}=1\,\,,$$
    $$K_iE_jK_i^{-1}=q^{\tl{B}_{ij}}E_j\,\,,\quad K_iF_jK_i^{-1}=q^{-\tl{B}_{ij}}F_j\,\,,\quad E_iF_j-(-1)^{s_is_j}\,F_jE_i=\D_{ij}\frac{K_i-K_i^{-1}}{q_i-q_i^{-1}}\,\,,$$
    $$E_i^2=F_i^2=0\ ,\ \text{ if } s_i=1\ ,$$ 
    $$E_iE_j-(-1)^{s_is_j}E_jE_i=0\ ,\quad F_iF_j-(-1)^{s_is_j}F_jF_i=0\ ,\ \text{ if } i\ne j\text{ and }\tl{B}_{ij}=0\ ,$$ 
    and higher order Serre relations\footnote{We do not give the Serre relations explicitly since they are not important for this text. However, we do check those relations in the modules we discuss.}, cf. Proposition 6.3.1 in \cite{Y99}. 
\end{dfn}
We note that in our presentation we set the value of the central parameter to zero as it is always zero in all finite-dimensional representations.

The algebra $U_q\tl{\mathfrak{g}}$ has a Hopf algebra structure, with comultiplication $\Delta$, given on the generators $E_i,F_i,K_i^{\pm 1}$, $i\in\tl{\mr{I}}$, by 
\bee{
\Delta(X)=X\otimes X\ ,\quad \text{for }X\in\{K_i, K_i^{-1}\vert \ i\in \tl{\mr{I}}\}\ ,}
\bee{
\Delta(X_i)=X_i\otimes K_i^{1/2}+K_i^{-1/2}\otimes X_i\ ,\quad \text{for }X_i\in \{E_i,F_i\}\ ,\ i\in\tl{\mr{I}}\ ,
}
Here we use the symmetric version of the coproduct which is convenient for our purposes. For that we extend the quantum affine algebra $U_q\tl{\mathfrak{g}}$ by $K_i^{1/2}$, $i\in \tl{\mr{I}}$, in the obvious way. We keep the same notation $U_q\tl{\mathfrak{g}}$ for the extended algebra and hope it does not lead to a confusion.

\begin{thm}[Theorem 6.6.1 in \cite{Y99}]
\label{thm:dynkin isomorphism}
    The quantum affine superalgebras $U_q\tl{\mathfrak{g}}$ defined in Definition \ref{DJ:super} are isomorphic as algebras for any choice of a Dynkin diagram for $\tl{\mathfrak{g}}$.
\qed
\end{thm}
Note that the isomorphisms in Theorem \ref{thm:dynkin isomorphism} are not Hopf algebra isomorphisms.

The Hopf subalgebra of $U_q\tl{\mathfrak{g}}$ generated by $K_i^{\pm1/2}$, $E_i$, $F_i$, $i\in\mr{I}$, is isomorphic to the quantum superalgebra $U_q\mathfrak{g}$ of finite type associated to $\mathfrak{g}$.

\subsection{New Drinfeld realization, finite-dimensional representations, and \texorpdfstring{$R$}--matrices}

It is well known that the finite-dimensional modules of the affine quantum groups are best studied using the new Drinfeld realization.
However, while the new Drinfeld realization is not difficult to guess except maybe the Serre relations, the isomorphism of this realization with the algebra $U_q\tl{\mathfrak{g}}$ defined in Section \ref{DJ:super} is not established up to now to the best of our knowledge. 

We write the (conjectured form of) new Drinfeld realization omitting the Serre relations. 

The quantum affine superalgebra $U_q\tl{\mathfrak{g}}$ (at level zero), corresponding to a choice of Dynkin diagram, is the unital associative algebra over $\mathbb{C}$ with generators  $K_i^{\pm}(z)=\sum_{\pm t\geq 0} K_{i,t}^\pm z^{-t}$, $X^{\pm}_i(z)=\sum_{t\in\Z} X^{\pm}_{i,t}z^{-t}$,  $i\in \mr{I}$, modulo relations:
$$[K_i^\pm(z),K_j^\pm(w)]=[K_i^\pm(z),K_j^\mp(w)]=0,\ \qquad K_{i,0}^+ K_{i,0}^-=1\ ,$$
$$(z-q^{\pm B_{ij}}w)\,K_i^\epsilon(z)X_j^\pm(w)=(q^{\pm B_{ij}}z-w)\,X_j^\pm(w)K_i^\epsilon(z)\ \text{ for }\epsilon=\pm\ ,$$
$$X_i^{\pm}(z)X_i^{\pm}(w)+X_i^{\pm}(w)X_i^{\pm}(z)=0,\quad \text{if }s_i=1\ ,$$
$$
X_i^\pm(z)X^\pm_j(w)= (-1)^{s_is_j}\,X_j^\pm(w)X_i^\pm(z), \qquad {\rm{if}}\ i\ne j \text{ and }B_{ij}=0\ ,
$$
$$(z-q^{\pm B_{ij}}w)\,X_i^{\pm}(z)X_j^\pm(w)=(-1)^{s_is_j}(q^{\pm B_{ij}}z-w)\,X_j^\pm(w)X_i^\pm(z)\ ,
$$ 
$$X_{i}^+(z)\,X_{j}^-(w)-(-1)^{s_is_j}X_j^-(w)\,X_i^+(z)=\D_{ij}\,\D\bigg(\frac{z}{w}\bigg)\,\frac{K_{i}^+(z)-K_{i}^-(z)}{q_i-q_i^{-1}}\ ,$$ 
where $\D(z)=\sum_{t\in\Z}z^t\,\in\C[\![z,z^{-1}]\!]\ .$
 
A vector $v$ in a $U_q\tl{\mathfrak{g}}$-module $V$ is called $\ell$-weighted if it is an eigenvector of $K_i^{\pm}(z)$, $i\in \mr{I}$. The collection of eigenvalues of $K_i^{+}(z)$ on $v$ is called the $\ell$-weight of $v$. In finite-dimensional modules all $\ell$-weights are $\mr{I}$-tuples of rational functions. We consider such tuples as a group with component-wise multiplication. The $q$-character of $V$ is a formal sum of the $\ell$-weights in $V$ with integer coefficients given by the dimensions of the corresponding $\ell$-weight spaces.

For $i\in \mr{I}, w\in \C^\times$, the $i$-th affine root is an $\mr{I}$-tuple of rational functions
\begin{align}\label{affine roots}
A_{i,w}(z)=\Big(\frac{q^{B_{ij}}z-w}{z-q^{B_{ij}}w}\Big)_{j\in \mr{I}}.
\end{align}

An $U_q\tl{\mathfrak{g}}$-module is called thin if multiplicity of each $\ell$-weight is 1.  If $V$ is thin then it has a basis of $\ell$-weight vectors. In this basis, if the matrix coefficient of $X_i^\pm(z)$ between two vectors is non-zero, then there exists $w\in\C^\times$, such that the ratio of the corresponding $\ell$-weights is the affine root $A_{i,w}^{\pm1}$. Moreover, this matrix coefficient is $c\delta(z/w)$, where $c$ is a constant. 

For the definition of the new Drinfeld realization and progress in types $\mathfrak{sl}(m\vert n)$ 
and $\mathfrak{osp}(m|n)$, see \cite{BFK24},  \cite{WLZ25}. For the definition of the new Drinfeld realization and a surjective map to $U_q\tl{\mathfrak{g}}$ in type D$_{2,1;\alpha}$, see \cite{HSTY08}.

The classification of finite-dimensional irreducible representations of $U_q\tl{\mathfrak{g}}$ is also not known in general. It is known in the case of $\mathfrak{g}=\mathfrak{sl}_{m\vert n}\ (m\ne n)$ and distinguished parity, see \cite{Zh96}\footnote{Th paper \cite{Zh96} is written for the Yangian case, the extension to quantum affine case is straightforward.}. 

\medskip

We prepare a few generalities. First, every $U_q\tl{\mathfrak{g}}$-module can be shifted by a non-zero complex number.

\begin{prp}
    For $z\in\mathbb{C}^\times$, there is a Hopf algebra automorphism $\tau_z$ of $U_q\tl{\mathfrak{g}}$ defined by: 
\bee{
\tau_z(E_0)=z\,E_0\ ,\quad \tau_z(F_0)=z^{-1}\,F_0\ ,\quad \tau_z(E_i)=E_i\ ,\quad \tau_z(F_i)=F_i\ , \quad \tau_z(K_j^{\pm1})=K_j^{\pm1},\quad i\in{\mr{I}}\ ,j\in\tl{\mr{I}}\ .
}
\p{
It is straightforward to check that $\tau_z$ is an isomorphism of Hopf algebras.
}
\end{prp}

For the new Drinfeld realization, $\tau_z(K_i^\pm(w))=K_i^\pm (w/z),$ $\tau_z(X_i^\pm(w))=X_i^\pm(w/z).$

Given a $U_q\tl{\mathfrak{g}}$-module $W$ and $z\in\mathbb{C}^\times$, we denote by $W(z)$ the pull-back of $W$ by $\tau_z$. Clearly, $W(z)$ is irreducible if and only if $W$ is irreducible.

\medskip 

Second, every $U_q\tl{\mathfrak{g}}$-module $W$ is clearly a $U_q{\mathfrak{g}}$-module  which is called the restriction of $W$.

Let $W$ be a $U_q\mathfrak{g}$-module.
Given $\lambda=(\lambda_1,\dots,\lambda_r)$, $\la_i\in\C$, define the subspace $W_\lambda\subset W$ of weight $\lambda$ to be
$$W_\lambda=\{w\in W:K_iw=q_i^{\lambda_i}w, \ i\in\mr{I}\}.$$
If $W_\lambda\ne0$, $\lambda$ is called a weight of $W$. A nonzero vector $w\in W_\lambda$ is called a vector of weight $\lambda$.

A vector $w\in W_\lambda$ of weight $\lambda$ is called a highest weight vector, or a singular vector, if $E_i\,w=0$, $i\in\mr{I}$.

Let $L_\la$ denote the irreducible representation of $U_q\mathfrak{g}$ with highest weight $\la$.

Each irreducible highest weight module of a
basic Lie superalgebra can be deformed to an irreducible highest weight module over the
corresponding Drinfeld-Jimbo algebra, see \cite{G07}. 

\medskip

Third, we discuss the $R$-matrices.

The quasi-triangular structure in the supersymmetric case is studied for all finite types in \cite{Y94}. We are not aware of any results in the affine cases.


Let $W$ be an irreducible $U_q\tl{\mathfrak{g}}$-module. Then it is easy to see that for generic $z$, the module $W(z)\otimes W$ is irreducible. Assuming the existence of new Drinfeld realization, this module is isomorphic to $W\otimes W(z)$. 
Let $\check{R}(z)$ be the intertwiner $$\check{R}(z):W(z)\otimes W\to W\otimes W(z).$$
Then $\check{R}(z)$ is unique up to a constant. One could fix the constant by choosing an appropriate highest weight vector $w\in W$ and demanding that $\check{R}(z) (w\otimes w)=w\otimes w$. Then $\check{R}(z)$ can be extended to all values of $z$ as a rational function of $z$. One expects that matrix $\check{R}(z)$ is obtained from the universal $R$-matrix by evaluating it at $W(z)\otimes W$ and multiplying the result by the permutation operator.

Such $\check{R}(z)$ satisfies the Quantum Yang-Baxter equation (QYBE). 

In this paper, we give expressions of $\check{R}(z)$, when $V$ is the non-trivial irreducible $U_q\tl{\mathfrak{g}}$-module of smallest dimension for all basic simple affine superalgebras, that is when $\tl{\mathfrak{g}}$ is of types $\wtl{\mathfrak{sl}}_{m\vert n}$, $\wtl{\mathfrak{osp}}_{m\vert 2n}$, $\tl{\text{D}}_{2\vert 1;\alpha}$, $\tl{\text{F}}_{3\vert 1}$, and $\tl{\text{G}}_{2\vert 1}$. In the last three cases (the exceptional types), these expressions are new to the best of our knowledge.

The appropriate limit $q\to 1$ is expected to give the rational $R$-matrix which is an intertwiner of the corresponding Yangian modules. We do not describe Yangians but simply describe the limit and give the rational $R$-matrices.

\section{Type \texorpdfstring{$\mathfrak{sl}_{m\vert n}\ (m>n\ge 1)$}{2}}
\label{super:sl}

In this section, we write the already known formula of $\check{R}(z)$ for the first fundamental representations of $U_q\wtl{\fk{sl}}_{m\vert n}$. In the matrix unit form, this $R$-matrix is called the Perk-Schultz matrix.

We consider the distinguished Dynkin diagram given by
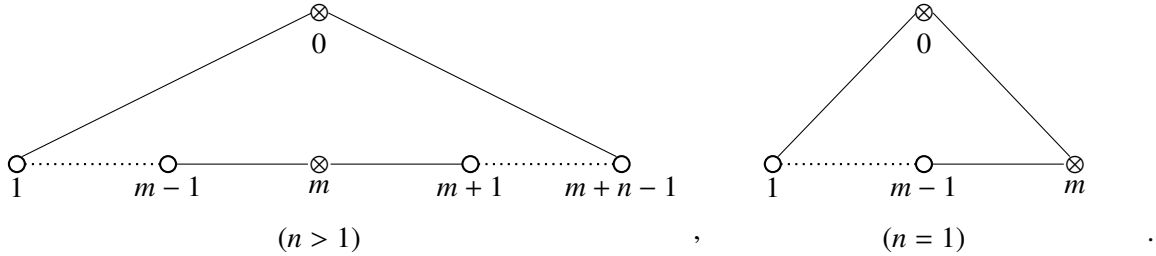
\begin{center}
    \begin{tikzpicture}[root/.style={circle, draw, thick, minimum size=6pt, inner sep=0pt}]
            \node[root] at (0,0) {};
            \node[root] at (2,0) {};
            \node at (4,0) {$\otimes$};
            \node[root] at (6,0) {};
            \node[root] at (8,0) {};
            \node at (4,2) {$\otimes$};
            \node at (0,-0.3) {$1$};
            \node at (2,-0.3) {$m-1$};
            \node at (4,-0.3) {$m$};
            \node at (6,-0.3) {$m+1$};
            \node at (8,-0.3) {$m+n-1$};
            \node at (4,1.6) {$0$};
            \draw[dotted, thick] (0.1,0) -- (1.9,0);
            \draw (2.1,0) -- (3.85,0);
            \draw (4.15,0) -- (5.9,0);
            \draw[dotted, thick] (6.1,0) -- (7.9,0);
            \draw (0.07,0.1) -- (3.9,2);
            \draw (4.12,2) -- (7.93,0.1);

            \node[root] at (10,0) {};
            \node[root] at (12,0) {};
            \node at (14,0) {$\otimes$};
            \node at (12,2) {$\otimes$};
            \node at (10,-0.3) {$1$};
            \node at (12,-0.3) {$m-1$};
            \node at (14,-0.3) {$m$};
            \node at (12,1.6) {$0$};
            \draw[dotted, thick] (10.1,0) -- (11.9,0);
            \draw (12.1,0) -- (13.85,0);
            \draw (10.07,0.1) -- (11.9,2);
            \draw (12.12,2) -- (13.93,0.1);

            \node at (4,-1) {$(n>1)$};
            \node at (12,-1) {$(n=1)$};
            \node at (9,-1) {,};
            \node at (15,-1) {.};
        \end{tikzpicture}
\end{center}

Let $\tl{L}_1(z)$ be the first fundamental representation of $U_q\wtl{\fk{sl}}_{m\vert n}$. When restricted to $U_q\fk{sl}_{m\vert n}$, we have $\tl{L}_1(z)\cong L_{\omega_1}$, where $\omega_1$ is the first fundamental weight. In general, let $\omega_i$ be the $i$-th fundamental weight, $1\le i\le m+n-1$. As $U_q\fk{sl}_{m\vert n}$-modules we have 
\eq{\label{super:tensorA}
\underbracket[0.1ex]{L_{\omega_1}}_{m+n}\otimes \underbracket[0.1ex]{L_{\omega_1}}_{m+n}\cong \underbracket[0.1ex]{L_{2\omega_1}}_{\binom{m+n+1}{2}}\oplus \underbracket[0.1ex]{L_{\omega_2}}_{\binom{m+n}{2}}\ .
}
Here and below, we use square under-brackets to indicate the dimensions of the modules. 

One can choose a basis $\{v_i:1\le i\le m+n\}$ for $L_{\omega_1}$ in the standard way, so that $v_1$ is an even non-zero highest weight vector and $F_iv_i=v_{i+1}$, $1\leq i\leq m+n-1$. Then, $v_1\otimes v_1$ is a singular vector of weight $2\omega_1$, and $q\,v_1\otimes v_2-v_2\otimes v_1$ is a singular vector of weight $\omega_2$. We generate respectively the modules $L_{2\omega_1}$ and $L_{\omega_2}$ using these singular vectors. Note that in the case when $m=1$, $v_2$ is an odd vector. 

For $\lambda=2\omega_1,\omega_2$, let $P_\lambda^q$ be the projector onto the $U_q\mathfrak{sl}_{m\vert n}$-module $L_\lambda$ in decomposition \eqref{super:tensorA}, and let $E_{ij}$ be matrix units corresponding to the basis of $v_i$, that is, $E_{ij}(v_k)=\D_{jk}v_i$.
For $1\le i\le m+n$, let
$$p_i=\begin{cases}
    0 & \text{if }i\le m\\
    1 & \text{if }i>m
\end{cases}\ ,\quad\quad q_i=\begin{cases}
    q & \text{if }i\le m\\
    q^{-1} & \text{if }i>m
\end{cases}\ .$$

\begin{thm}[\cite{PS81}, \cite{GZ03}]
In terms of projectors, we have
\eq{\label{super:R proj A}
\check{R}(z)=P_{2\omega_1}^q-q^{-2}\frac{1-q^{2}z}{1-q^{-2}z}P_{\omega_2}^q\ .
}
In terms of matrix units, we have
\eq{\label{super:RqA}
\check{R}(z)=\ & \sum_{i=1}^{r+1}\,\frac{q_i-q_i^{-1}z}{q-q^{-1}z}\,(-1)^{p_i}\, E_{ii}\otimes E_{ii}\,+\,\sum_{i<j}\,\frac{z(q_i-q_i^{-1})}{q-q^{-1}z}\,(-1)^{p_ip_j}\,E_{ii}\otimes E_{jj} \\
& +\,\sum_{i>j}\,\frac{q_i-q_i^{-1}}{q-q^{-1}z}\,(-1)^{p_ip_j}\,E_{ii}\otimes E_{jj} +\,\frac{1-z}{q-q^{-1}z}\,\sum_{i\ne j}(-1)^{p_ip_j}\,E_{ij}\otimes E_{ji}\ .
}
\p{There are many ways to prove the
projector formula in \eqref{super:R proj A}.  For example, one could use the method of \cite{DM25a} and the $q$-characters. Alternatively,
this is multiplicity free case, and one could use the tensor graph method in \cite{GZ03}. 

The matrix unit formula in \eqref{super:RqA} is the quantized version of the Perk-Schulz matrix in \cite{PS81}.
}
\end{thm}

Let $P_\lambda=\lim_{q\to 1}P_{\lambda}^q\,\,$ be the $U\fk{sl}_{m|n}$ projector, let $I$ be the identity operator, and let $P$ be the graded flip operator given by 
$$P=\sum_{i,j=1}^{m+n}\,(-1)^{p_ip_j}\,E_{ij}\otimes E_{ji}\,=P_{2\omega_1}-P_{\omega_2}\ .$$

\begin{cor}[\cite{PS81}]
In the rational case, the corresponding rational $R$-matrix is given by 
\eq{\label{super:RuA}
\check{R}(u)=P_{2\omega_1}+\frac{1+u}{1-u}P_{\omega_2}=\frac{1}{1-u}(I-uP)\ .
}
\end{cor}
\begin{proof}
We substitute $z=q^{2u}$ in \eqref{super:R proj A} and \eqref{super:RqA} and take the limit $q\to 1$.
\end{proof}

\section{Type \texorpdfstring{$\mathfrak{osp}_{m\vert 2n}\ (m\geq 3, n\geq 1)$}{2}}
\label{super:osp}

In this section, we write the already known formula of $\check{R}(z)$ for the first fundamental representation of $U_q\wtl{\fk{osp}}_{m\vert 2n}$.

We consider the distinguished Dynkin diagram given by 

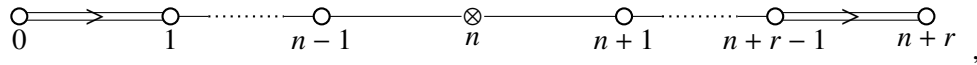
\begin{center}
    \begin{tikzpicture}[root/.style={circle, draw, thick, minimum size=6pt, inner sep=0pt}]
            \node[root] at (-2,0) {};
            \node[root] at (0,0) {};
            \node[root] at (2,0) {};
            \node at (4,0) {$\otimes$};
            \node[root] at (6,0) {};
            \node[root] at (8,0) {};
            \node[root] at (10,0) {};
            \node at (-2,-0.3) {$0$};
            \node at (0,-0.3) {$1$};
            \node at (2,-0.3) {$n-1$};
            \node at (4,-0.3) {$n$};
            \node at (6,-0.3) {$n+1$};
            \node at (8,-0.3) {$n+r-1$};
            \node at (10,-0.3) {$n+r$};
            \draw (0.1,0) -- (0.5,0);
            \draw (1.5,0) -- (1.9,0);
            \draw[dotted, thick] (0.5,0) -- (1.5,0);
            \draw (2.1,0) -- (3.85,0);
            \draw (4.15,0) -- (5.9,0);
            \draw (6.1,0) -- (6.5,0);
            \draw (7.5,0) -- (7.9,0);
            \draw[dotted, thick] (6.5,0) -- (7.5,0);
            \draw (8.1,0.05) -- (9.9,0.05);
            \draw (8.1,-0.05) -- (9.9,-0.05);
            \draw (-1.9,0.05) -- (-0.1,0.05);
            \draw (-1.9,-0.05) -- (-0.1,-0.05);
            \node at (9,0) {\scalebox{1.1}{$>$}};
            \node at (-1,0) {\scalebox{1.1}{$>$}};
        \end{tikzpicture}\ ,
\end{center}
when $m=2r+1$ is odd, and
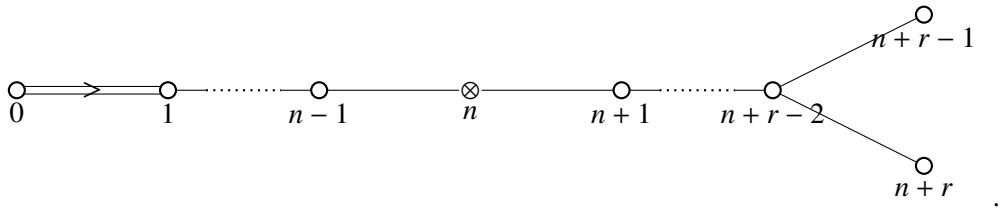
\begin{center}
    \begin{tikzpicture}[root/.style={circle, draw, thick, minimum size=6pt, inner sep=0pt}]
            \node[root] at (-2,0) {};
            \node[root] at (0,0) {};
            \node[root] at (2,0) {};
            \node at (4,0) {$\otimes$};
            \node[root] at (6,0) {};
            \node[root] at (8,0) {};
            \node[root] at (10,1) {};
            \node[root] at (10,-1) {};
            \node at (-2,-0.3) {$0$};
            \node at (0,-0.3) {$1$};
            \node at (2,-0.3) {$n-1$};
            \node at (4,-0.3) {$n$};
            \node at (6,-0.3) {$n+1$};
            \node at (8,-0.3) {$n+r-2$};
            \node at (10,0.7) {$n+r-1$};
            \node at (10,-1.3) {$n+r$};
            \draw (0.1,0) -- (0.5,0);
            \draw (1.5,0) -- (1.9,0);
            \draw[dotted, thick] (0.5,0) -- (1.5,0);
            \draw (2.1,0) -- (3.85,0);
            \draw (4.15,0) -- (5.9,0);
            \draw (6.1,0) -- (6.5,0);
            \draw (7.5,0) -- (7.9,0);
            \draw[dotted, thick] (6.5,0) -- (7.5,0);
            \draw (8.1,0.05) -- (9.9,0.95);
            \draw (8.1,-0.05) -- (9.9,-0.95);
            \draw (-1.9,0.05) -- (-0.1,0.05);
            \draw (-1.9,-0.05) -- (-0.1,-0.05);
            \node at (-1,0) {\scalebox{1.1}{$>$}};
        \end{tikzpicture}\ .
\end{center}
when $m=2r$ is even.

Let $\tl{L}_1(z)$ be the first fundamental representation of $U_q\wtl{\fk{osp}}_{m\vert 2n}$. Restricting to $U_q\fk{osp}_{m\vert 2n}$, we have $\tl{L}_1(z)\cong L_{\omega_1}$, where $\omega_1$ is the first fundamental weight. In general, let $\omega_i$ be the $i$-th fundamental weight, $1\le i\le n+r$. As $U_q\fk{osp}_{m\vert 2n}$-modules we have 
\eq{\label{super:tensorB}
\underbracket[0.1ex]{L_{\omega_1}}_{m+2n}\otimes \underbracket[0.1ex]{L_{\omega_1}}_{m+2n}\cong \underbracket[0.1ex]{L_{2\omega_1}}_{\binom{m+2n+1}{2}}\oplus \underbracket[0.1ex]{L_{\omega_2}}_{\binom{m+2n}{2}-1}\oplus \underbracket[0.1ex]{L_{\omega_0}}_{1}\ ,\quad \text{if }m\ne 2n\ .
}
When $m=2n$, the last two summands combine to form an indecomposable summand $W$.

For $\lambda=2\omega_1,\omega_2,\omega_0$, when $m\ne 2n$, let $P_\lambda^q$ be the projector onto the $U_q\fk{osp}_{m\vert 2n}$-module $L_\lambda$ in the decomposition \eqref{super:tensorB}. When $m=2n$, let $P_W^q$ be the projector onto the indecomposable summand $W$.

\begin{thm}[\cite{MDGL05}]
In terms of projectors, we have
\eq{\label{super:R proj B}
\check{R}(z)=\begin{cases}
    P_{2\omega_1}^q-q^{-2}\dfrac{1-q^{2}z}{1-q^{-2}z}P_{\omega_2}^q-q^{m-2n-2}\dfrac{1-q^{2n-m+2}z}{1-q^{m-2n-2}z}\,P_{\omega_0}^q & \text{if }m\ne 2n\ , 
    \vspace{10pt}\\ 
    P_{2\omega_1}^q-q^{-2}\dfrac{1-q^{2}z}{1-q^{-2}z}P_{W}^q & \text{if }m=2n\ .
    \end{cases}
}
\p{
The projector formulas  \eqref{super:R proj B} can be computed using the method of \cite{DM25a} or the tensor graph method in \cite{GZ03}. These formulas are given in \cite{MDGL05}.
}
\end{thm}

Let $P_\lambda=\lim_{q\to 1}P_{\lambda}^q\,\,$ and $P_W=\lim_{q\to 1}P_W^q$, be the $U\fk{osp}_{m\vert 2n}$ projectors.

\begin{cor}
In the rational case, the corresponding rational $R$-matrix is given by 
\eq{\label{super:RuB}
\check{R}(u) = \begin{cases}
    P_{2\omega_1}+\dfrac{1+u}{1-u}P_{\omega_2}+\dfrac{2n-m+2+2u}{2n-m+2-2u}P_{\omega_0} & \text{if }m\ne 2n\ , 
    \vspace{10pt}\\
    P_{2\omega_1}+\dfrac{1+u}{1-u}P_W & \text{if }m=2n\ . 
    \end{cases}
}
\end{cor}
\begin{proof}
We substitute $z=q^{2u}$ in \eqref{super:R proj B} and take the limit $q\to 1$.
\end{proof}

\section{Type D\texorpdfstring{$_{2\vert 1;\alpha}$}{2}}
\label{super:D}

In this section, we give an explicit formula of $\check{R}(z)$ for the $18$-dimensional representation of $U_q\tl{\text{D}}_{2\vert 1;\A}$ in terms of $U_q\text{D}_{2\vert 1;\A}$-linear maps and projectors related to the tensor square of the $17$-dimensional representation of $U_q\text{D}_{2\vert 1;\A}$. 

\subsection{Quantum algebra of finite type}

We consider the following Dynkin diagram for D$_{2\vert 1;\A}$, where all nodes are ferminonic. 
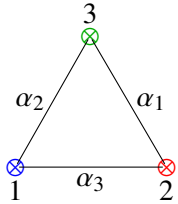
\begin{center}
    \begin{tikzpicture}[scale=0.5]
        \tikzmath{\a = 2; \b = 0.025; \e=0.1;}
        \node at (-\b,\a*1.732) {$\textcolor{mygreen}{\otimes}$};
        \node at (-\a,0) {$\textcolor{blue}{\otimes}$};
        \node at (\a,0) {$\textcolor{red}{\otimes}$};
        \node at (-0.7*\a-0.2,\a*0.866) {$\alpha_2$};
        \node at (0.7*\a+0.2,0.866*\a) {$\alpha_1$};
        \node at (0,-0.4) {$\alpha_3$};

        \node at (\a, -0.6) {$2$};
        \node at (-\a, -0.6) {$1$};
        \node at (-\b, \a*1.732+0.6) {$3$};

        \draw[-](\a-\e,0) -- (-\a+\e,0);
        \draw[-] (\a-\e/2,0+\e*1.732/2) -- (-\b+\e/2,\a*1.732-\e*1.732/2);
        \draw[-](-\b-\e/2,\a*1.732-\e*0.866) -- (-\a+\e/2,0+\e*0.866);
    \end{tikzpicture}
\end{center}
Here $\A_1,\A_2,\A_3\in\mathbb{C}$ are such that $\A_i\ne 0,-1$, and $\A_1+\A_2+\A_3=0$. We set $q_i=q^{\alpha_i}$, $i=1,2,3$.
The corresponding Cartan matrix is given by 
$$\begin{bmatrix}
    0 & \A_3 & \A_2 \\
    \A_3 & 0 & \A_1 \\
    \A_2 & \A_1 & 0
\end{bmatrix}\ .$$
Since this is already a symmetric matrix, we have $d_1=d_2=d_3=1$ in this case.

The Lie superalgebra $\mathfrak{g}=$D$_{2\vert 1;\alpha}$ has a $9$-dimensional even part isomorphic to a direct sum of three copies of $\mathfrak{sl}_2$, that is, $\mathfrak{g}_{\ol{0}}\cong \mathfrak{g}_1\oplus\mathfrak{g}_2\oplus\mathfrak{g}_3$, where $\mathfrak{g}_i\cong\mathfrak{sl}_2$, $i=1,2,3$. The $8$-dimensional odd part is given by $\mathfrak{g}_{\ol{1}}\cong V_1\otimes
V_2\otimes V_3$, where each $V_i$, $i=1,2,3$, is the $2$-dimensional vector representation of $\mathfrak{g}_i$ and the trivial module for two other $\mathfrak{sl}_2$.
The super Lie bracket gives a $\mathfrak{g}_{\ol{0}}$-module homomorphism from the symmetric square of $\mathfrak{g}_{\ol{1}}$ onto $\mathfrak{g}_{\ol{0}}$. The space of such maps is labeled by three constants $\alpha_1,\A_2,\A_3\in\C$.
In addition, by the super-Jacobi identity, one gets $\A_1+\A_2+\A_3=0$. Rescaling the generators of  $\mathfrak{g}_{\ol{1}}$,   one can choose $\A_1=1$, $\A_2=\A$ and $\A_3=-1-\A$. In this way, one obtains the superalgebra $\mathfrak{g}$ depending on a complex parameter $\A$. If $\A \neq 0,-1$, this algebra is simple. When $\A=1,-2,-1/2$, this algebra is isomorphic to $\mathfrak{osp}(4|2)$.

The irreducible modules for D$_{2\vert 1;\alpha}$ are classified (see \cite{K75}, \cite{J85} for details) by triples $(a,b,c)$ with $a,b,c\in\Z_{\ge 1}$ or $a=b=c=0$. The numbers $a,b,c$ correspond to the weights of the singular vector with respect to each of the three copies of $\mathfrak{sl}_2$. 
The corresponding highest weight, with respect to the chosen all fermionic Dynkin diagram, is given by $a\omega_1+b\omega_2+c\omega_3$, where 
$$\omega_1=\frac{1}{2}(-\alpha_1,\alpha_1,\alpha_1)\ ,\quad \omega_2=\frac{1}{2}(\alpha_2,-\alpha_2,\alpha_2)\ ,\quad \omega_3=\frac{1}{2}(\alpha_3,\alpha_3,-\alpha_3)\ .$$

We denote the $U_q$D$_{2\vert 1;\alpha}$-module corresponding to the triple $(a,b,c)$ by $L_{(a,b,c)}$. 

We construct a weighted basis $\{v_i\}_{i=1}^{17}$ for the $17$-dimensional $U_q$D$_{2\vert 1;\alpha}$-module $L_{(1,1,1)}$ as follows. The highest weight vector $v_1$ is of highest weight $(-\alpha_1,-\alpha_2,-\alpha_3)$. Other $v_i$ are such that the action of $U_q$D$_{2\vert 1;\alpha}$ is as in Figure \ref{fig:D e-f action}. In Figure \ref{fig:D e-f action},  the action of $E_1,E_2,E_3$ is shown in the left diagram and the action of $F_1,F_2,F_3$ in the right diagram. We use blue color for $E_1$ or $F_1$, red for $E_2$ or $F_2$, and green for $E_3$ or $F_3$, as in the Dynkin diagram nodes. Here $\overline{i}=18-i$. The coefficients shown on arrows are quantum numbers and if a coefficient is not shown then it is one. For example, $E_1\,v_{10}=-[\A_2]\,v_7$, $F_2\,v_8=[\alpha_3]\,v_{\overline{6}}$, $F_3\,v_8=[\alpha_2]\,v_{\overline{5}}$, etc. In the diagram on the right, the triples of numbers are the weights with respect of $U_{q_i}\mathfrak{sl}_2$, $i=1,2,3$, inside $U_q$D$_{2\vert 1;\alpha}$.
\begin{figure}[ht!]
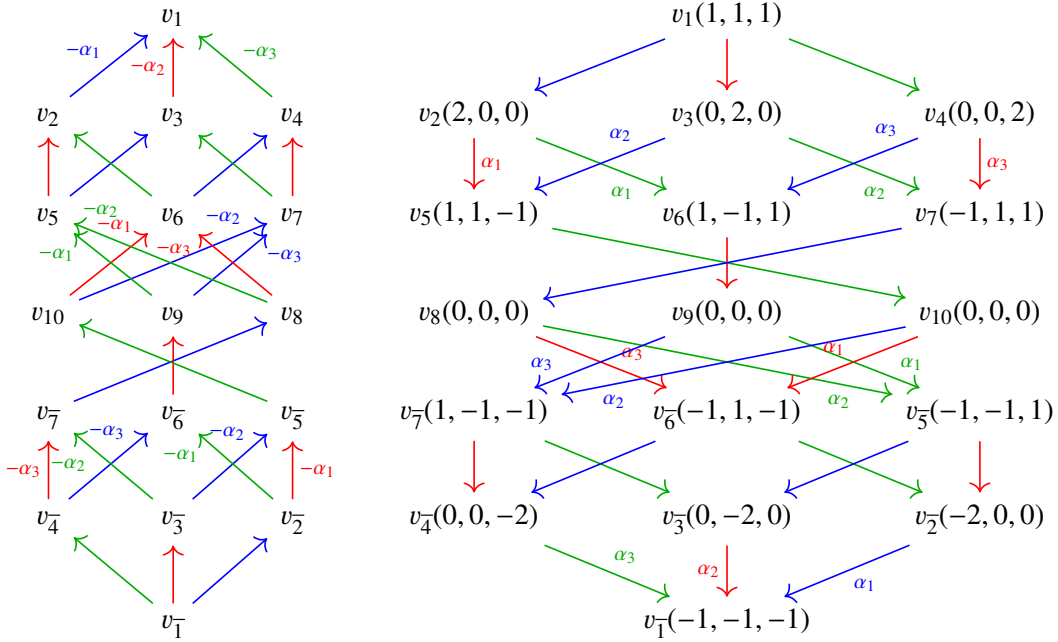

    \centering
\dia{
\& v_1 \& \& \& v_1(1,1,1)\ar[ld, blue]\ar[d, red]\ar[rd, mygreen] \& \\
v_{2}\ar[ru, blue, "-\A_1"] \& v_{3}\ar[u, red, "-\A_2"] \& v_{4}\ar[lu, mygreen, "-\A_3"'] \& v_{2}(2,0,0)\ar[rd, mygreen,"\A_1"',pos=0.8]\ar[d,red,"\A_1"] \& v_3(0,2,0)\ar[ld, blue, "\A_2"', pos=0.2]\ar[rd, mygreen, "\A_2"', pos=0.8] \& v_4(0,0,2)\ar[ld, blue, "\A_3"', pos=0.1] \ar[d, red, "\A_3", pos=0.5] \\
v_5\ar[u, red]\ar[ru, blue] \& v_6\ar[lu, mygreen]\ar[ru, blue] \& v_7\ar[lu, mygreen]\ar[u, red] \& v_5(1,1,-1)\ar[rrd, mygreen] \& v_6(1,-1,1)\ar[d, red] \& v_7(-1,1,1)\ar[lld, blue]\\
v_{10}\ar[ru, red, "-\A_1", pos=0.9]\ar[rru, blue, "-\A_2", pos=0.9] \& v_9\ar[lu, mygreen, "-\A_1", pos=0.9]\ar[ru, blue, "-\A_3"', pos=0.9] \& v_8\ar[llu, mygreen, "-\A_2"', pos=0.99]\ar[lu, red, "-\A_3", pos=0.99] \& v_8(0,0,0)\ar[rrd, mygreen, "\A_2"', pos=0.9] \ar[rd, red, "\A_3", pos=0.6] \& v_9(0,0,0)\ar[ld, blue, "\A_3"', pos=0.8] \ar[rd, mygreen, "\A_1",pos=0.8] \& v_{10}(0,0,0)\ar[lld, blue, "\A_2", pos=0.9] \ar[ld, red,"\A_1"',] \\
v_{\overline{7}}\ar[rru, blue] \& v_{\overline{6}}\ar[u, red] \& v_{\overline{5}}\ar[llu, mygreen] \& v_{\overline{7}}(1,-1,-1)\ar[rd, mygreen]\ar[d, red] \& v_{\overline{6}}(-1,1,-1)\ar[rd, mygreen]\ar[ld, blue] \& v_{\overline{5}}(-1,-1,1)\ar[d, red] \ar[ld, blue] \\
v_{\overline{4}}\ar[ru, blue, "-\A_3", pos=0.8]\ar[u, red, "-\A_3"] \& v_{\overline{3}}\ar[ru, blue, "-\A_2", pos=0.8]\ar[lu, mygreen, "-\A_2", pos=0.75] \& v_{\overline{2}}\ar[u, red, "-\A_1"']\ar[lu, mygreen, "-\A_1", pos=0.9] \& v_{\overline{4}}(0,0,-2)\ar[rd, mygreen, "\A_3"] \& v_{\overline{3}}(0,-2,0)\ar[d, red, "\A_2"'] \& v_{\overline{2}}(-2,0,0)\ar[ld, blue, "\A_1"] \\
\& v_{\overline{1}}\ar[ru, blue]\ar[u, red]\ar[lu, mygreen] \& \& \& v_{\overline{1}}(-1,-1,-1) \&
}
\caption{The $17$-dimensional module of $U_q$D$_{2\vert 1;\alpha}$}
    \label{fig:D e-f action}
\end{figure}

The parity of $v_1, v_5, v_6, v_7, v_{\bar 5}, v_{\bar 6},v_{\bar 7},v_{\bar 1}$, is $1$, other basic vectors are even. After a $q\to 1$ limit, the module $L_{(1,1,1)}$ is the adjoint module of D$_{2\vert 1;\A}$. In our chosen basis, up to the $q\to 1$ limit, the vectors $v_2$, $v_{\ol{2}}$ along with some linear combination of $v_9,v_{10}$ form a basis of the adjoint module of $\fk{sl}_2$. The same is true for $v_3$, $v_{\ol{3}}$, along with some linear combination of $v_8,v_{10}$, and for $v_4$, $v_{\ol{4}}$, along with some linear combination of $v_8,v_9$. Together, these vectors form a basis of the even part of D$_{2\vert 1;\A}$. Up to a $q\to 1$ limit, the vectors $v_1, v_5, v_6,v_7, v_{\bar 5}, v_{\bar 6},v_{\bar 7},v_{\bar 1}$, form the $8$-dimensional representation of $\fk{sl}_2\oplus \fk{sl}_2 \oplus \fk{sl}_2$, and they give a basis of the odd part of D$_{2\vert 1;\alpha}$.

Note that unlike \cite{DM25a}, \cite{DM25b} the bases we chose in this paper are not orthonormal with respect to Shapovalov form. That simplifies the constants in the action but the $R$-matrix below is not self-adjoint.

\begin{prp}
   The basis $\{v_i\}_{i=1}^{17}$ exists. 
\end{prp}
\begin{proof}
    We explicitly check all the relations of $U_q$D$_{2\vert 1;\alpha}$.
\end{proof}

\subsection{Quantum affine algebra and \texorpdfstring{$q$}--characters}

We consider the Dynkin diagram for $U_q\tl{\text{D}}_{2\vert 1;\alpha}$ as shown below, where all nodes are fermionic.
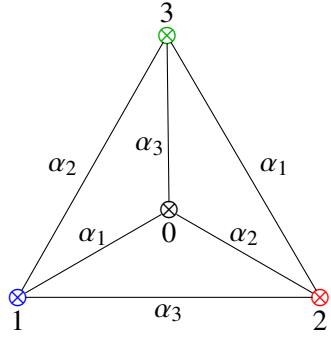
\begin{center}
    \begin{tikzpicture}
        \tikzmath{\a = 2; \b = 0.025; \e=0.1;}
        \node at (0,\a/1.732) {$\textcolor{black}{\otimes}$};
        \node at (-\b,\a*1.732) {$\textcolor{mygreen}{\otimes}$};
        \node at (-\a,0) {$\textcolor{blue}{\otimes}$};
        \node at (\a,0) {$\textcolor{red}{\otimes}$};
        \node at (-0.7*\a,\a*0.866) {$\alpha_2$};
        \node at (0.7*\a,0.866*\a) {$\alpha_1$};
        \node at (0,-0.2) {$\alpha_3$};
        \node at (0.5*\a,0.4*\a) {$\alpha_2$};
        \node at (-0.5*\a,0.4*\a) {$\alpha_1$};
        \node at (-0.25,\a) {$\alpha_3$};

        \node at (\a, -0.3) {$2$};
        \node at (-\a, -0.3) {$1$};
        \node at (0, \a/1.732-0.3) {$0$};
        \node at (-\b, \a*1.732+0.3) {$3$};

        \draw[-](0+\e,\a/1.732-\e/1.732) -- (\a-\e,0+\e/1.732);
        \draw[-](0-\e,\a/1.732-\e/1.732) -- (-\a+\e,0+\e/1.732);
        \draw[-](\a-\e,0) -- (-\a+\e,0);
        \draw[-] (-\b,\a*1.732-\e) -- (0,\a/1.732+\e);
        \draw[-] (\a-\e/2,0+\e*1.732/2) -- (-\b+\e/2,\a*1.732-\e*1.732/2);
        \draw[-](-\b-\e/2,\a*1.732-\e*0.866) -- (-\a+\e/2,0+\e*0.866);
    \end{tikzpicture}
\end{center}
The corresponding affine Cartan matrix is given by 
$$\begin{bmatrix}
    0 & \A_1 & \A_2 & \A_3 \\
    \A_1 & 0 & \A_3 & \A_2 \\
    \A_2 & \A_3 & 0 & \A_1 \\
    \A_3 & \A_2 & \A_1 & 0
\end{bmatrix}\ .$$
Since this is already a symmetric matrix, we have $d_0=d_1=d_2=d_3=1$ in this case.

The affine roots, see \eqref{affine roots}, are given by 
\bee{
        A_{1,a} = \mr{2}_{aq_3}^{aq_3^{-1}}\,\mr{3}_{aq_2}^{aq_2^{-1}} \ ,\quad 
        A_{2,a} = \mr{1}_{aq_3}^{aq_3^{-1}}\,\mr{3}_{aq_1}^{aq_1^{-1}} \ ,\quad 
        A_{3,a} = \mr{1}_{aq_2}^{aq_2^{-1}}\,\mr{2}_{aq_1}^{aq_1^{-1}}\ .
}
Here $\mr{1}_b^a$, $\mr{2}_b^a$, $\mr{3}_b^a$ are triples of rational functions given by 
$$\mr{1}_b^a=\bigg(\sqrt{a/b}\ \frac{1-z/a}{1-z/b},1,1\bigg)\ ,\quad \mr{2}_b^a=\bigg(1,\sqrt{a/b}\ \frac{1-z/a}{1-z/b},1\bigg) \ ,\quad \mr{3}_b^a=\bigg(1,1,\sqrt{a/b}\ \frac{1-z/a}{1-z/b}\bigg)\ ,\quad a,b\in\C^\times\ .$$ 

A thin $18$-dimensional $U_q\tl{\text{D}}_{2\vert 1;\alpha}$-module  $\tl{L}_{(1,1,1)}$, was constructed in \cite{FJM22} using new Drinfeld realization. 
The diagram in Figure \ref{fig:supp 18 dim D} shows the $q$-character of $\tl{L}_{(1,1,1)}$ and the action in a corresponding $\ell$-weighted basis. In the diagram we use the notation $\mr {1}_{c,d}^{a,b}=\mr {1}_c^a \mr{1}_d^b$, $\mr {2}_{c,d}^{a,b}=\mr {2}_c^a \mr{2}_d^b$, etc.
The straight arrows of color $i$ correspond to multiplication in the $q$-character by $A_{i,q_i^{-1}}^{-1}$, while the squiggly arrows correspond to multiplication by $A_{i,q_i}^{-1}$. The matrix coefficients of $X_i^\pm(z)$ between corresponding $\ell$-weight vectors are delta functions $\delta(q_iz)$ or $\delta(q_i^{-1}z)$ for straight and squiggly  arrows correspondingly, multiplied by constants. The constants are called the supplement in \cite{FJM22} and  are shown in the picture. The constants on the right of the arrows (when viewing the arrows from tail towards head) correspond to the action of $X_i^+(z)$, while the constants on the left correspond to the action of $X_i^-(z)$.

\begin{figure}[ht!]
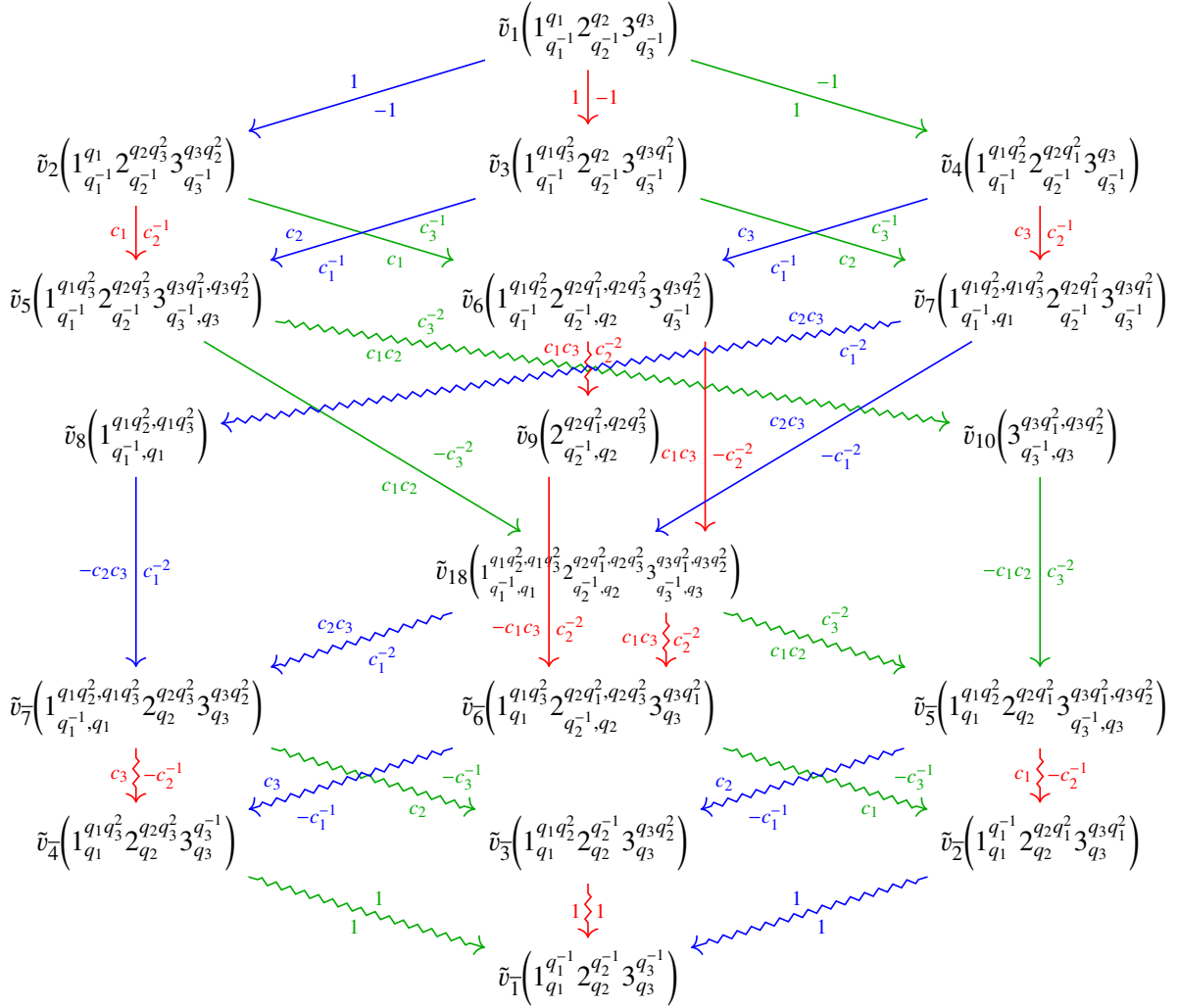

    \centering
\ddia{column sep=2cm, row sep=0.75cm}{
\& \tl v_1\bigg(\mr{1}^{q_1}_{q_1^{-1}}\mr{2}^{q_2}_{q_2^{-1}}\mr{3}^{q_3}_{q_3^{-1}}\bigg) \ar[ld,blue,"-1","1"']\ar[d,red,"-1","1"']\ar[rd,mygreen,"-1","1"'] \& \\
\tl v_2\bigg(\mr{1}^{q_1}_{q_1^{-1}}\mr{2}^{q_2q_3^2}_{q_2^{-1}}\mr{3}^{q_3q_2^2}_{q_3^{-1}}\bigg) \ar[d,red,"c_2^{-1}","c_1"']\ar[rd,mygreen,"c_3^{-1}","c_1"',pos=0.8] \& \tl v_3\bigg(\mr{1}^{q_1q_3^2}_{q_1^{-1}}\mr{2}^{q_2}_{q_2^{-1}}\mr{3}^{q_3q_1^2}_{q_3^{-1}}\bigg)                  \ar[ld,blue,"c_1^{-1}","c_2"',pos=0.8]\ar[rd,mygreen,"c_3^{-1}","c_2"',pos=0.8] \& \tl v_4\bigg(\mr{1}^{q_1q_2^2}_{q_1^{-1}}\mr{2}^{q_2q_1^2}_{q_2^{-1}}\mr{3}^{q_3}_{q_3^{-1}}\bigg) \ar[ld,blue,"c_1^{-1}","c_3"',pos=0.8]\ar[d,red,"c_2^{-1}","c_3"'] \\
\tl v_5\bigg(\mr{1}^{q_1q_3^2}_{q_1^{-1}}\mr{2}^{q_2q_3^2}_{q_2^{-1}}\mr{3}^{q_3q_1^2,q_3q_2^2}_{q_3^{-1},q_3}\bigg) \ar[rdd,mygreen,"-c_3^{-2}","c_1c_2"',pos=0.7]\ar[rrd,mygreen,squiggly,"c_3^{-2}","c_1c_2"',pos=0.2] \& \tl v_6\bigg(\mr{1}^{q_1q_2^2}_{q_1^{-1}}\mr{2}^{q_2q_1^2,q_2q_3^2}_{q_2^{-1},q_2}\mr{3}^{q_3q_2^2}_{q_3^{-1}}\bigg) \ar[d,red,squiggly,"c_2^{-2}","c_1c_3"',pos=0.15] \ar[dd,red,shift left=45pt,"-c_2^{-2}","c_1c_3"',pos=0.6] \& \tl v_7\bigg(\mr{1}^{q_1q_2^2,q_1q_3^2}_{q_1^{-1},q_1}\mr{2}^{q_2q_1^2}_{q_2^{-1}}\mr{3}^{q_3q_1^2}_{q_3^{-1}}\bigg) \ar[lld,blue,squiggly,"c_1^{-2}","c_2c_3"',pos=0.1]\ar[ldd,blue,"-c_1^{-2}","c_2c_3"',pos=0.5] \\
\tl v_8\bigg(\mr{1}^{q_1q_2^2,q_1q_3^2}_{q_1^{-1},q_1}\bigg) \ar[dd,blue,"c_1^{-2}","-c_2c_3"'] \& \tl v_9\bigg(\mr{2}^{q_2q_1^2,q_2q_3^2}_{q_2^{-1},q_2}\bigg) \ar[dd,red,shift right=15pt,"c_2^{-2}","-c_1c_3"',pos=0.8] \& \tl v_{10}\bigg(\mr{3}^{q_3q_1^2,q_3q_2^2}_{q_3^{-1},q_3}\bigg) \ar[dd,mygreen,"c_3^{-2}","-c_1c_2"'] \\
\& \tl v_{18}\mathsmaller{\bigg(\mr{1}^{q_1q_2^2,q_1q_3^2}_{q_1^{-1},q_1}\mr{2}^{q_2q_1^2,q_2q_3^2}_{q_2^{-1},q_2}\mr{3}^{q_3q_1^2,q_3q_2^2}_{q_3^{-1},q_3}\bigg)} \ar[ld,blue,squiggly,"c_1^{-2}","c_2c_3"']\ar[d,red,squiggly, shift left=30pt,"c_2^{-2}","c_1c_3"']\ar[rd,mygreen,squiggly,"c_3^{-2}","c_1c_2"'] \& \\
\tl v_{\ol{7}}\bigg(\mr{1}^{q_1q_2^2,q_1q_3^2}_{q_1^{-1},q_1}\mr{2}^{q_2q_3^2}_{q_2}\mr{3}^{q_3q_2^2}_{q_3}\bigg) \ar[d,red,squiggly,"-c_2^{-1}","c_3"']\ar[rd,mygreen,squiggly,"-c_3^{-1}","c_2"',pos=0.8] \& \tl v_{\ol{6}}\bigg(\mr{1}^{q_1q_3^2}_{q_1}\mr{2}^{q_2q_1^2,q_2q_3^2}_{q_2^{-1},q_2}\mr{3}^{q_3q_1^2}_{q_3}\bigg) \ar[ld,blue,squiggly,"-c_1^{-1}","c_3"',pos=0.8]\ar[rd,mygreen,squiggly,"-c_3^{-1}","c_1"',pos=0.8]  \& \tl v_{\ol{5}}\bigg(\mr{1}^{q_1q_2^2}_{q_1}\mr{2}^{q_2q_1^2}_{q_2}\mr{3}^{q_3q_1^2,q_3q_2^2}_{q_3^{-1},q_3}\bigg) \ar[ld,blue,squiggly,"-c_1^{-1}","c_2"',pos=0.8]\ar[d,red,squiggly,"-c_2^{-1}","c_1"'] \\
\tl v_{\ol{4}}\bigg(\mr{1}^{q_1q_3^2}_{q_1}\mr{2}^{q_2q_3^2}_{q_2}\mr{3}^{q_3^{-1}}_{q_3}\bigg) \ar[rd,mygreen,squiggly,"1","1"'] \& \tl v_{\ol{3}}\bigg(\mr{1}^{q_1q_2^2}_{q_1}\mr{2}^{q_2^{-1}}_{q_2}\mr{3}^{q_3q_2^2}_{q_3}\bigg) \ar[d,red,squiggly,"1","1"'] \& \tl v_{\ol{2}}\bigg(\mr{1}^{q_1^{-1}}_{q_1}\mr{2}^{q_2q_1^2}_{q_2}\mr{3}^{q_3q_1^2}_{q_3}\bigg) \ar[ld,blue,squiggly,"1","1"'] \\
\& \tl v_{\ol{1}}\bigg(\mr{1}^{q_1^{-1}}_{q_1}\mr{2}^{q_2^{-1}}_{q_2}\mr{3}^{q_3^{-1}}_{q_3}\bigg) \&
}
\caption{The $q$-character and the supplement for $18$-dimensional $U_q\tl{\text{D}}_{2|1;\A}$-module $\tl{L}_{(1,1,1)}(1)$. Here $c_i=q_i-q_i^{-1}$}
\label{fig:supp 18 dim D}
\end{figure}

We now construct the module $\tl{L}_{(1,1,1)}(z)$, $z\in\C^\times$, in the Drinfeld-Jimbo realization.

The restriction of $\tl{L}_{(1,1,1)}(z)$ to $U_q$D$_{2\vert 1;\alpha}$ decomposes as
$$\tl{L}_{(1,1,1)}(z)\cong L_{(1,1,1)}\oplus L_{(0,0,0)}\ .$$ 

To give action of $U_q\tl{\text{D}}_{2\vert 1;\alpha}$  on $\tl{L}_{(1,1,1)}(z)$ in the Drinfeld-Jimbo realization, 
we use a basis $\{v_{i}\}_{i=1}^{17}\cup\{v_{18}\}$ for $\tl{L}_{(1,1,1)}(z)$, where $\{v_i\}_{i=1}^{17}$ is the basis of $U_q$D$_{2\vert 1;\alpha}$-module $L_{(1,1,1)}$ as chosen in Figure \ref{fig:D e-f action} and $v_{18}$ is a basis of the $1$-dimensional $U_q$D$_{2\vert 1;\alpha}$-module $L_{(0,0,0)}$. Note that  $v_i$ are scalar multiples of $\tilde v_i$ for $i=1,2,\dots,7$ and $i=\bar 1,\bar 2,\dots,\bar 7$ but not for $i=8,9,10$.

Let
\eq{\label{D18 E0 action}
E_0 = z\,\Big(& -\frac{[2]_{\A_1}}{2}\,E_{8,1}-\frac{[2]_{\A_2}}{2}\,E_{9,1}-\frac{[2]_{\A_3}}{2}\,E_{10,1}+\frac{1}{2}\,E_{18,1} - [\A_1]\,E_{\overline{1},8} -[\A_2]\,E_{\overline{1},9} - [\A_3]\,E_{\overline{1},10} \\ 
& -[2]^{\mr{i}}_{\A_2}\,[2]^{\mr{i}}_{\A_3}\,E_{\overline{1},18} - [\A_1]\,E_{\overline{7},2} - [\A_2]\,E_{\overline{6},3}-[\A_3]\,E_{\overline{5},4}+E_{\overline{4},5}+E_{\overline{3},6}+E_{\overline{2},7}\Big) \ ,
}
\eq{\label{D18 F0 action}
F_0 = z^{-1}\,\Big(& [\A_1]\,E_{1,8} + [\A_2]\,E_{1,9} + [\A_3] E_{1,10} + [2]^{\mr{i}}_{\A_2}\,[2]^{\mr{i}}_{\A_3}\,E_{1,18}-\frac{[2]_{\A_1}}{2}\,E_{8,\overline{1}}-\frac{[2]_{\A_2}}{2}\,E_{9,\overline{1}} \\
& -\frac{[2]_{\A_3}}{2}\,E_{10,\overline{1}} + \frac{1}{2}\,E_{18,\overline{1}} + E_{2,\overline{7}} + E_{3,\overline{6}} + E_{4,\overline{5}} + [\A_3]\,E_{5,\overline{4}} + [\A_2]\,E_{6,\overline{3}} + [\A_1]\,E_{7,\overline{2}}\Big) \ ,
}
where $E_{ij}$ are matrix units, that is, $E_{ij}v_k=\D_{jk}v_i$.

\begin{prp}
    The action of $E_i, F_i$, $i=1,2,3,$ as in Figure  \ref{fig:D e-f action} and $E_0, F_0$ as in \eqref{D18 E0 action}, \eqref{D18 F0 action}  gives a $U_q\tl{\text{D}}_{2\vert 1;\alpha}$-module structure.
\end{prp}
\begin{proof}
    We check all the relations in $U_q\tl{\text{D}}_{2\vert 1;\alpha}$ by a direct computation.
\end{proof}

\subsection{The \texorpdfstring{$R$}--Matrix}

Recall that the socle of a representation $W$ is the sum of all its irreducible subrepresentations. We denote it by $\soc(W)$.

Recall that the under-brackets indicate the dimension of the corresponding irreducible $U_q$D$_{2\vert 1;\alpha}$-module. 
\begin{thm}\label{square D thm}
    As $U_q$D$_{2\vert 1;\alpha}$-modules we have
\eq{\label{D tensor2}
\big(\tl{L}_{(1,1,1)}\big)^{\otimes 2} & \cong \big(\underbracket[0.1ex]{L_{(1,1,1)}}_{17}\oplus \underbracket[0.1ex]{L_{(0,0,0)}}_{1}\big)^{\otimes 2} \\
& \cong \big(\underbracket[0.1ex]{L_{(1,1,1)}}_{17} + \underbracket[0.1ex]{L_{(2,2,2)}}_{110} + \underbracket[0.1ex]{L_{(1,1,1)}}_{17}\big)\oplus \underbracket[0.1ex]{L_{(3,1,1)}}_{48} \oplus \underbracket[0.1ex]{L_{(1,3,1)}}_{48} \oplus \underbracket[0.1ex]{L_{(1,1,3)}}_{48} \oplus\ 2\underbracket[0.1ex]{L_{(1,1,1)}}_{17} \oplus\ 2\underbracket[0.1ex]{L_{(0,0,0)}}_{1} \ .
}

Let $W_a$ be the $144 =17+110+17$ dimensional direct summand in \eqref{D tensor2}. We have 
$$\soc(W_a)\cong L_{(1,1,1)}\ ,\quad  \soc\big(W_a/\soc(W_a)\big)\cong L_{(2,2,2)}\ .
$$

We have $$\dim\en_{U_q\text{D}_{2\vert 1;\alpha}}\Big(\big(\tl L_{(1,1,1)}\big)^{\otimes 2}\Big)=4+2^2+2^2+3+2=17 \ .$$
\end{thm}
\p{
By straightforward calculations we observe  inside the tensor square of $U_q$D$_{2\vert 1;\alpha}$-module $L_{(1,1,1)}\oplus L_{(0,0,0)}$ the following.

There are $9$ linearly independent singular vectors, which we enumerate as $u_1$, $u_{2}$, $u_{3}$, $u_{4}$, $u_{5a}$, $u_{5b}$, $u_{5c}$, $u_{6a}$, $u_{6b}$. 
The singular vector $u_1=v_1\otimes v_1$ generates the $127$-dimensional indecomposable submodule $L_{(2,2,2)}+L_{(1,1,1)}:=U_a$, inside the indecomposable direct summand $W_a=\big(L_{(1,1,1)}+L_{(2,2,2)}+L_{(1,1,1)}\big)$, in \eqref{D tensor2}. The singular vector 
\bee{
u_{5a} =\ & [2]_{\A_1}\,(v_1\otimes v_8-v_8\otimes v_1)+[2]_{\A_2}\,(v_1\otimes v_9-v_9\otimes v_1)+[2]_{\A_3}\,(v_1\otimes v_{10}-v_{10}\otimes v_1) \\
& -2\big(q_1^{1/2}v_2\otimes v_7-q_1^{-1/2}v_7\otimes v_2\big) -2\big(q_2^{1/2}v_3\otimes v_6-q_2^{-1/2}v_6\otimes v_3\big)-2\big(q_3^{1/2}v_4\otimes v_5-q_3^{-1/2}v_5\otimes v_4\big)\ ,
}
generates the $17$-dimensional submodule $L_{(1,1,1)}:=V_{a}$, inside $U_a$. Thus we have $V_a\subset U_a \subset W_a$. The singular vectors $u_{2}$, $u_3$, $u_4$, given respectively by 
$$u_{2}=v_1\otimes v_2+q_1\,v_2\otimes v_1\ ,\quad u_{3}=v_1\otimes v_3+q_2\,v_3\otimes v_1\ ,\quad u_{4}=v_1\otimes v_4+q_3\,v_4\otimes v_1\ ,$$
generate the $48$-dimensional irreducible direct summands $L_{(3,1,1)}$, $L_{(1,3,1)}$, $L_{(1,1,3)}$ respectively, in \eqref{D tensor2}. The singular vectors $u_{5b}=v_1\otimes v_{18}$ and $u_{5c}=v_{18}\otimes v_1$ generate the two irreducible direct summands $L_{(1,1,1)}$ in \eqref{D tensor2}, which we denote by $V_{b}$, $V_{c}$ respectively. The last two direct summands are one-dimensional, each generated by $u_{6a}, u_{6b}$. We have  
\bee{
u_{6a}=\ & 2[\A_2]\,[\A_3]\,\big(-v_1\otimes v_{\overline{1}} + v_{\ol{1}}\otimes v_1 + v_2\otimes v_{\overline{2}}+ v_{\ol{2}}\otimes v_2+ v_5\otimes v_{\overline{5}}- v_{\ol{5}}\otimes v_5+ v_6\otimes v_{\overline{6}}-v_{\ol{6}}\otimes v_6+ v_7\otimes v_{\overline{7}} \\
& -v_{\ol{7}}\otimes v_7\big) + 2[\A_1]\,[\A_3]\,(v_3\otimes v_{\overline{3}} + v_{\ol{3}}\otimes v_3) + 2[\A_1]\,[\A_2]\,(v_4\otimes v_{\overline{4}} + v_{\ol{4}}\otimes v_4) + \big([\A_1]\big)^2v_8\otimes v_8 \\
& + \big([\A_2]\big)^2v_9\otimes v_9 + \big([\A_3]\big)^2v_{10}\otimes v_{10} - [\A_1][\A_2]\,(v_8\otimes v_9+v_9\otimes v_8) - [\A_2][\A_3]\,(v_9\otimes v_{10}+v_{10}\otimes v_9) \\
& - [\A_1][\A_3]\,(v_8\otimes v_{10}+v_{10}\otimes v_8)\ ,
}
and $u_{6b}=v_{18}\otimes v_{18}$.

Thus, the submodule generated by all singular vectors has codimension $17$. In addition, there exists a vector $u_0$ of weight $(-\alpha_1,-\alpha_2,-\alpha_3)$ given by 
\bee{
u_{0} =\ & [2]_{\A_1}^{\mr{i}}\,\big(q_1\,v_1\otimes v_8 + q_1^{-1}\,v_8\otimes v_1 + q_2\,v_1\otimes v_9 + q_2^{-1}\,v_9\otimes v_1 + q_3^{-1}\,v_1\otimes v_{10} \\
& + q_3\,v_{10}\otimes v_1 - q_1^{\frac{1}{2}}\,v_2\otimes v_7 - q_1^{-\frac{1}{2}}v_7\otimes v_2\big) -\frac{[2]_{\A_2}\,[\A_1]}{[\A_2/2]}\,(v_3\otimes v_6-v_6\otimes v_3) \\
& + \frac{[2]_{\A_2}\,[\A_1]}{[\A_3]}\,\big(q_3^{-\frac{1}{2}}v_4\otimes v_5-q_3^{\frac{1}{2}}v_5\otimes v_4\big) - \frac{[2]_{\A_2-\A_1}\,[\A_1]}{[\A_3]}\,\big(q_3^{\frac{1}{2}}v_4\otimes v_5-q_3^{-\frac{1}{2}}v_5\otimes v_4\big) \ ,
}
such that $u_0$ is not in the submodule generated by singular vectors and such that $u_0$ generates a $144$ dimensional module. 

Note that we chose here the overall factor $[2]_{\A_1}^{\mr{i}}$ breaking the symmetry between $\alpha_i$. This factor makes the $q\to 1$ limit of $u_0$ well-defined.

Finally we clearly have $$\dim\en_{U_q\text{D}_{2\vert 1;\alpha}} W_a=2, \quad 
\dim \operatorname{Hom}_{U_q\text{D}_{2\vert 1;\alpha}} (W_a,V_b\oplus V_c)+\dim \operatorname{Hom}_{U_q\text{D}_{2\vert 1;\alpha}} (V_b\oplus V_c, W_a)=2+2=4$$ 
and the total space of endomorphisms of $\big(\tl L_{(1,1,1)}\big)^{\otimes 2}$ has dimension $17$. 
}

\begin{rmk}
    See \cite{GM25}, where
    modules with multiplicity-free socles are studied in the context of finite-dimensional representations of quantum affine $\mathfrak{sl}_2$. There is a natural directed graph corresponding to such modules. For the module $V_t$, the graph is  $L_{(1,1,1)}\to L_{(2,2,2)} \to L_{(1,1,1)}$. There are two nontrivial proper directed subgraphs corresponding to the two submodules $V_1 $ and $V_a$.
    
\end{rmk}
The $R$-matrix is a map of $U_q$D$_{2\vert 1;\alpha}$-modules. According to the last part of Theorem \ref{square D thm} any such map can be written as a linear combination of $17$ maps. Next, we choose those $17$ maps.

In  decomposition \eqref{D tensor2}, let $P_W^q$ be the projector onto the indecomposable summand $W_a$, let $P^q_{(3,1,1)}$, $P^q_{(1,3,1)}$, $P^q_{(1,1,3)}$ be projectors onto $L_{(3,1,1)}$, $L_{(1,3,1)}$, $L_{(1,1,3)}$, respectively. We identify $V_b\oplus V_c$ with $\C^2\otimes L_{(1,1,1)}$ and let $\id_{2\times 2}\otimes P^q_{(1,1,1)}$ be the projector in \eqref{D tensor2} onto the direct sum $V_b\oplus V_{c}$. In the same way, let $\id_{2\times 2}\otimes P^q_{(0,0,0)}$ be the projector onto $\C^2\otimes L_{(0,0,0)}$. Let $P^q_{0\to a}$ be the unique $U_q$D$_{2\vert 1;\alpha}$-linear map that maps all direct summands in \eqref{D tensor2}  to zero except $W_a$ which is mapped to the submodule $V_a$. The map $P^q_{0\to a}$  sends $u_0$ to $u_{5a}$ and all other singular vectors to zero. Similarly, let $P^q_{0\to b}$, $P^q_{0\to c}$ be the unique $U_q$D$_{2\vert 1;\alpha}$-linear maps that send $W_a$ onto the direct summands $V_b$, $V_c$, respectively, mapping $u_0$ to $u_{5b}$, $u_{5c}$, respectively and annihilating all other singular vectors. Let $P^q_{b\to a}$, $P^q_{c\to a}$ be the unique $U_q$D$_{2\vert 1;\alpha}$-linear maps from $V_b$, $V_c$, respectively, onto the submodule $V_a$, mapping $u_{5b}$, $u_{5c}$, respectively, to $u_{5a}$, and annihilating all other singular vectors.

\begin{thm}
\label{superD:Rcheck}
Let $d(z)=(1-q_1^{2}z)(1-q_2^{2}z)(1-q_3^{2}z)$. The $R$-matrix $\check{R}(z)$ is given by 
\eq{\label{thm:RqD21x}
\check{R}(z) =\ & P^q_{W} - q_1^2\frac{1-q_1^{-2}z}{1-q_1^2z}\,P^q_{(3,1,1)} - q_2^{2}\frac{1-q_2^{-2}z}{1-q_2^{2}z}\,P^q_{(1,3,1)} - q_3^{2}\frac{1-q_3^{-2}z}{1-q_3^{2}z}\,P^q_{(1,1,3)} \\
& -\dfrac{[2]_{\A_1}^{\mr{i}}\,(z-1)(2z^2-d(1)\,z-2)}{2\,d(z)} P^q_{0\to a} - \dfrac{[2]_{\A_1}^{\mr{i}}\,d(-1)\,z(z-1)}{2\,d(z)} \big(P^q_{0\to b} + P^q_{0\to c}\big) \\
& + \dfrac{d(1)\,z(z-1)}{2\,d(z)} \big(P^q_{b\to a} + P^q_{c\to a} \big) + \frac{f_1(z)}{2\,d(z)}\otimes P^q_{(1,1,1)} + \frac{f_0(z)}{2\,d(z)}\otimes P^q_{(0,0,0)}\ ,
}
where the matrices $f_{1}(z)$ and $f_{0}(z)$ are as follows:
$$f_{1}(z)=\begin{bmatrix}
 d(1)\,z(z+1) & (z-1)(2z^2-\B\,z+2)
\vspace{0.25cm}\\
(z-1)(2z^2-\B\,z+2) & d(1)\,z(z+1)
\end{bmatrix}\ ,$$
\medskip
$$f_{0}(z)=\begin{bmatrix} (z-1)(2z^2-\B\,z+2) & \big([2]_{\A_1}^{\mr{i}}\big)^2\,d(1)\,z(z+1) \vspace{0.25cm} \\ \dfrac{d(1)}{([2]_{\A_1}^{\mr{i}})^2}\,z(z+1) & (z-1)(2z^2-\B\,z+2) \end{bmatrix}\ .$$

\medskip

Here $\B=d(-1)-4$.
\p{
We use the equation $$\check{R}(z)\,\Delta(F_0)=\Delta(F_0)\,\check{R}(z):\tl L_{(1,1,1)}(z)\otimes\tl{L}_{(1,1,1)}(1)\to \tl L_{(1,1,1)}(1)\otimes\tl{L}_{(1,1,1)}(z)\ ,$$ with $F_0$ in \eqref{D18 F0 action}. Equivalently, one can use $E_0$ in \eqref{D18 E0 action}.
}
\end{thm}

\begin{rmk}
The operator $f_1(z)$ above shows the action of $\check{R}(z)$ on the two modules $L_{(1,1,1)}$ that appear as direct summands in \eqref{D tensor2}. If one incorporates in $f_1(z)$ the action of $\check{R}(z)$ on all modules $L_{(1,1,1)}$ appearing in \eqref{D tensor2}, then it can be organized as a $4\times 4$ matrix given by 
$$\tl f_{1}(z)=\frac{1}{2\,d(z)}\begin{bmatrix}
 1 & 0 & 0 & 0
\vspace{0.25cm} \\
-[2]_{\A_1}^{\mr{i}}(z-1)(2z^2-d(1)\,z-2) & 1 & d(1)\,z(z-1) & d(1)\,z(z-1)
\vspace{0.25cm} \\
-[2]_{\A_1}^{\mr{i}}d(-1)\,z(z-1) & 0 & d(1)\,z(z+1) & (z-1)(2z^2-\B\,z+2)
\vspace{0.25cm}\\
-[2]_{\A_1}^{\mr{i}}d(-1)\,z(z-1) & 0 & (z-1)(2z^2-\B\,z+2) & d(1)\,z(z+1)
\end{bmatrix}\ ,$$
where the chosen ordered basis is $\{u_0,u_{5a},u_{5b},u_{5c}\}$. Note that $u_0$ corresponds to a quotient module and therefore can be modified by adding any multiple of $u_{5a}$. Such a change does not change the  matrix  $\tilde f_1(z)$. However, the projection to the $17\times 4$-dimensional subspace corresponding to the four copies of $L_{(1,1,1)}$ is not a $U_q$D$_{2\vert 1;\alpha}$ homomorphism.
\end{rmk}

\bigskip 

The tensor product $\tl{L}_{(1,1,1)}(z)\otimes\tl{L}_{(1,1,1)}(1)$ is thin and therefore irreducible, except for some special values of $z$. These special values of $z$, and the corresponding submodules, are listed below. Note that the submodules for $z=q^k$, are quotient modules for $z=q^{-k}$.
\begin{enumerate}
    \item $z=q_1^2$: We have a $66$-dimensional submodule isomorphic to $L_{(3,1,1)}\oplus L_{(1,1,1)}\oplus L_{(0,0,0)}$ as a $U_q$D$_{2\vert 1;\alpha}$-module, and generated by $u_2$.
    \item $z=q_1^{-2}$: We have a $258$-dimensional submodule isomorphic to $W_a\oplus L_{(1,3,1)}\oplus L_{(1,1,3)}\oplus L_{(1,1,1)}\oplus L_{(0,0,0)}$ as a $U_q$D$_{2\vert 1;\alpha}$-module, and generated by $u_1$.
    \item $z=q_2^{2}$: We have a $66$-dimensional submodule isomorphic to $L_{(1,3,1)}\oplus L_{(1,1,1)}\oplus L_{(0,0,0)}$ as a $U_q$D$_{2\vert 1;\alpha}$-module, and generated by $u_3$.
    \item $z=q_2^{-2}$: We have a $258$-dimensional submodule isomorphic to $W_a\oplus L_{(3,1,1)}\oplus L_{(1,1,3)}\oplus L_{(1,1,1)}\oplus L_{(0,0,0)}$ as a $U_q$D$_{2\vert 1;\alpha}$-module, and generated by $u_1$.
    \item $z=q_3^{2}$: We have a $66$-dimensional submodule isomorphic to $L_{(1,1,3)}\oplus L_{(1,1,1)}\oplus L_{(0,0,0)}$ as a $U_q$D$_{2\vert 1;\alpha}$-module, and generated by $u_4$.
    \item $z=q_3^{-2}$: We have a $258$-dimensional submodule isomorphic to $W_a\oplus L_{(3,1,1)}\oplus L_{(1,3,1)}\oplus L_{(1,1,1)}\oplus L_{(0,0,0)}$ as a $U_q$D$_{2\vert 1;\alpha}$-module, and generated by $u_1$.
    \item $z=1$: The tensor product is completely reducible with a $1$-dimensional submodule generated by $u_{6a}-u_{6b}$, and a $323$-dimensional submodule generated by $u_1$.
\end{enumerate}

\begin{rmk}
    We note that $\check{R}(1)\ne\id$ in this case. More precisely, $\check{R}(1)$ acts as identity on all summands (equivalently, on all singular vectors) in \eqref{D tensor2}, except the last summand of multiplicity two, where it swaps the two singular vectors of weight $(0,0,0)$.  In particular, the module $(\tl{L}_{(1,1,1)}(z))^{\otimes 2}$ is reducible and, thus, the module $\tl{L}_{(1,1,1)}(z)$ is an imaginary module.
\end{rmk}

We consider the $q\to 1$ rational limit. With our normalization of the singular vectors $u_1,\dots, u_{6b}$, the limit of these vectors exists, and therefore the limit of the projectors and of the $17$ chosen $U_q$D$_{2\vert 1;\A}$-linear maps exists. Let $P_W$, $P_{(a_1,a_2,a_3)}$ be the  $q\to 1$  limit of the $U_q$D$_{2\vert 1;\alpha}$-projectors $P_W^q$, $P_{(a_1,a_2,a_3)}^q$ respectively, and let $P_{x\to y}$ be the $q\to 1$ limit of the $U_q$D$_{2\vert 1;\alpha}$-linear maps $P_{x\to y}^q$ ($x=0,b,c$ and $y=a,b,c$), appearing in the expression \eqref{thm:RqD21x} of $\check{R}(z)$.

\begin{cor}
    Let $\bar{d}(u)=(\A_1+u)(\A_2+u)(\A_3+u)$. The rational $R$-matrix $\check{R}(u)$ for the Yangian of D$_{2\vert 1;\alpha}$ is given by 
    \eq{\label{superD:raional R matrix}
    \check{R}(u)=\ & P_W+\frac{\alpha_1-u}{\alpha_1+u}P_{(3,1,1)}+\frac{\alpha_2-u}{\alpha_2+u}P_{(1,3,1)}+\frac{\alpha_3-u}{\alpha_3+u}P_{(1,1,3)} + \frac{2\A_1 u^2}{\bar d(u)}P_{0\to a} \\
    & + \frac{4\A_1^2 u}{\bar d(u)}\big(P_{0\to b} + P_{0\to c}\big) + \frac{\A_1\A_2\A_3}{2\,\bar d(u)}\big(P_{b\to a}+P_{c\to a}\big) + \frac{\bar f_1(u)}{2\,\bar d(u)}\otimes P_{(1,1,1)}+ \frac{\bar f_2(u)}{2\,\bar d(u)}\otimes P_{(0,0,0)}\ ,
    }
    where the matrices $\bar f_1(u)$ and $\bar f_0(u)$ are as follows:
    $$\bar f_{1}(u)=\begin{bmatrix}
     2\,\A_1\A_2\A_3 & u(\A_1^2+\A_2^2+\A_3^3-2u^2)
    \vspace{0.25cm}\\
    u(\A_1^2+\A_2^2+\A_3^3-2u^2) & 2\,\A_1\A_2\A_3
    \end{bmatrix}\ ,$$
    \medskip
    $$\bar f_{0}(u)=\begin{bmatrix} u(\A_1^2+\A_2^2+\A_3^3-2u^2) & 2\,\A_1\A_2\A_3 \vspace{0.25cm}\\ 2\,\A_1\A_2\A_3 & u(\A_1^2+\A_2^2+\A_3^3-2u^2) \end{bmatrix}\ .$$
\p{
We substitute $z=q^{2u}$ in \eqref{thm:RqD21x} and take an appropriate $q\to 1$ limit. Apriori, the $q\to 1$ limit of $\check{R}(q^{2u})$ does not exist. To have a limit we conjugate $\check{R}(q^{2u})$ by $A\otimes A$ where the operator $A:V\to V$ is given by $A\,v_i=v_i$ for $1\le i\le 17$ and $A\,v_{18}=\frac{1}{[2]_{\A_1}^{\mr{i}}}v_{18}$. Thus, we obtain $\check{R}(u)=\lim_{q\to 1}(A\otimes A)\check{R}(q^{2u})(A\otimes A)^{-1}$ as above in \eqref{superD:raional R matrix}. Note that conjugation by an operator of the form $A\otimes A$ preserves the quantum Yang-Baxter equation and, for our choice of  $A$, $A\otimes A$ commutes with action of  $U_q{\text{D}}_{2|1;\A}$ but not with action of $U_q\tl{\text{D}}_{2|1;\A}$.
}
\end{cor}

\section{Type F\texorpdfstring{$_{3\vert 1}$}{2}}
\label{super:F}

In this section, we give the expression of $\check{R}(z)$ for a $41$-dimensional representation of $U_q\tl{\text{F}}_{3\vert 1}$ in terms of projectors related to the tensor square of the $40$-dimensional representation of $U_q$F$_{3\vert 1}$. This $R$-matrix  has the form similar to the $R$-matrices of untwisted and twisted types in cases of non-trivial multiplicity.

\subsection{Quantum algebra of finite type}

We consider the distinguished Dynkin diagram where one node is fermionic and the other three bosonic, as shown below:

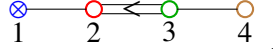
\begin{center}
\begin{tikzpicture}[root/.style={circle, draw, thick, minimum size=6pt, inner sep=0pt}]
    \node at (1,0) {$\textcolor{blue}{\otimes}$};
    \node[root, color=red] at (2,0) {};
    \node[root, color=mygreen] at (3,0) {};
    \node[root, color=brown] at (4,0) {};

    \draw (1.1,0) -- (1.9,0);
    \node at (2.5,0) {\scalebox{1.1}{$<$}};
    \draw (2.9,0.06)--(2.1,0.06);
    \draw (2.9,-0.06)--(2.1,-0.06);
    \draw (3.1,0) -- (3.9,0);

    \node at (1,-0.3) {1};
    \node at (2,-0.3) {2};
    \node at (3,-0.3) {3};
    \node at (4,-0.3) {4};
\end{tikzpicture}\ .
\end{center}
The corresponding Cartan matrix is given by 
$$\begin{bmatrix}
    0 & 1 & 0 & 0 \\
    -1 & 2 & -2 & 0 \\
    0 & -1 & 2 & -1 \\
    0 & 0 & -1 & 2
\end{bmatrix}\ .$$
We choose $d_1=-1, d_2=1, d_3=d_4=2$ in this case.

The Lie superalgebra F$_{3\vert 1}$ has a $24$-dimensional even part isomorphic to the direct sum $\mathfrak{sl}_2\oplus\mathfrak{so}_7$. The $16$-dimensional odd part of F$_{3\vert 1}$ is isomorphic to the tensor product of $2$-dimensional vector representation of $\mathfrak{sl}_2$ and $8$-dimensional spin representation of $\mathfrak{so}_7$.

The finite dimensional irreducible modules of F$_{3\vert 1}$ are classified (see \cite{K75}, \cite{M13} for details) by quadruples $(a,b,c,d)$ of non-negative integers such that $k=(2a-3b-4c-2d)/3$ is a non-negative integer different from $1$, and if $k=0$ then $(a,b,c,d)=(0,0,0,0)$, if  $k=0$ then $b=d=0$, if $k=3$ then $b=2d+1$. 

 We denote the irreducible $U_q$F$_{3\vert 1}$-module with highest weight $(a,b,c,d)$, by $L_{(a,b,c,d)}$. 

We construct a weighted basis $\{v_i\}_{i=1}^{40}$ for the $40$-dimensional $U_q$F$_{3\vert 1}$-module $L_{(3,0,0,0)}$ as follows. The highest weight vector $v_1$ is of weight $(3,0,0,0)$. Other vectors $v_i$ are such that the action of $U_q$F$_{3\vert 1}$ is as in Figure \ref{fig:F4 e-f action}. Figure \ref{fig:F4 e-f action} shows the action of $F_i$, $i=1,2,3,4$. We use blue for $F_1$, red for $F_2$, green for $F_3$ and brown for $F_4$, as in the Dynkin diagram nodes. Here $\overline{i}=41-i$. The coefficients are quantum numbers and if a coefficient is not shown then it is one. For example, $F_2\,v_{20}=[2]\,v_{\overline{7}}$, $F_3\,v_{21}=[2]_2\,v_{\overline{9}}$, etc. The action of $E_i$, $i=1,2,3,4$, is symmetric with respect to the action of $F_i$ and can be read from the diagram obtained by reflecting the diagram shown in Figure \ref{fig:F4 e-f action} about a horizontal line passing through the weight zero vectors $v_{19},v_{20},v_{21},v_{22}$. For example, $E_2\,v_{20}=[2]\,v_7$, $E_3\,v_{21}=[2]_2\,v_9$, etc. Specifically, when $1\le j,k\le 18$, we have $E_i\,v_j=a\,v_k$ if and only if $F_i\,v_{\overline{k}}=a\,v_{\overline{j}}$, and when $23\le j,k\le 40$ we have $E_i\,v_j=(-1)^{s_i}a\,v_k$ if and only if $F_i\,v_{\overline{k}}=a\,v_{\overline{j}}$. In the case, when $19\le j\le 22$ and $k=7,9,10,18$, we have $E_i\,v_j=a\,v_k$ if and only if $F_i\,v_j=a\,v_{\overline{k}}$, and $E_i\,v_{\ol{k}}=(-1)^{s_i}a\,v_{j}$ if and only if $F_i\,v_k=a\,v_j$. In addition, we choose the numbering of the vectors $v_1,\dots,v_{40}$ in such a way that $v_{11},\dots,v_{18}$ and $v_{\overline{18}},\dots,v_{\overline{11}}$ are vectors of odd parity, and the rest of the vectors are of even parity. 

\begin{prp}
    The basis $\{v_i\}_{i=1}^{40}$ exists.
\end{prp}
\begin{proof}
    We expicitly check all the relations of $U_q$F$_{3\vert 1}$. 
\end{proof}

After a $q\to 1$ limit, the module $L_{(3,0,0,0)}$ is the adjoint module for F$_{3\vert 1}$. In our chosen basis, up to the $q\to 1$ limit, the vectors $v_1$, $v_{19}$, $v_{\overline{1}}$ can be identified with a basis of the $\mathfrak{sl}_2$ subalgebra of the even part of F$_{3\vert 1}$, while the vectors $v_2,\dots,v_{10}$, $v_{20},v_{21},v_{22}$, $v_{\overline{10}},\dots,v_{\overline{2}}$ form a basis of the $\mathfrak{so}_7$ subalgebra of the even part. The vectors $v_{11},\dots,v_{18}$ form the spin representation of $\mathfrak{so}_7$, and so do the vectors $v_{\overline{18}},\dots, v_{\overline{11}}$. Together they form a representation of $\mathfrak{sl}_2\oplus\mathfrak{so}_7$, namely the $2$-dimensional vector representation of $\mathfrak{sl}_2$ tensor the spin representation of $\mathfrak{so}_7$. This tensor product can be identified with the odd part of the superalgebra F$_{3\vert 1}$.

\begin{figure}[ht!]
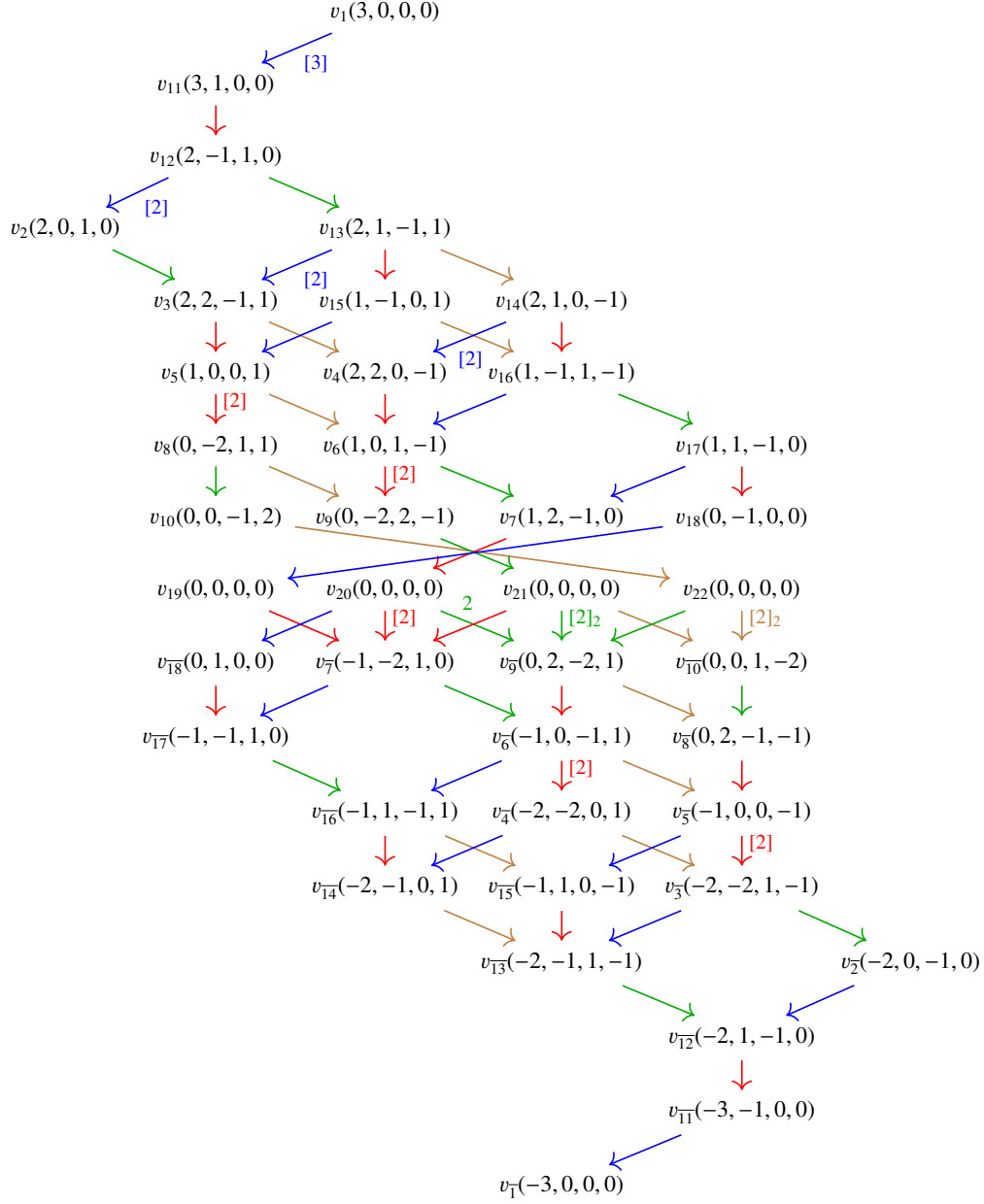

    \centering
\ddia{row sep=12pt, column sep=-1pt, font=\footnotesize}{
\& \& v_1(3,0,0,0)\ar[ld,blue,"{[3]}"] \& \& \& \\
\& v_{11}(3,1,0,0)\ar[d,red] \& \& \& \& \\
\&  v_{12}(2,-1,1,0)\ar[rd,mygreen]\ar[ld,blue,"{[2]}"] \& \& \& \& \\
v_2(2,0,1,0)\ar[rd,mygreen] \& \& v_{13}(2,1,-1,1)\ar[d,red]\ar[rd,brown]\ar[ld,blue,"{[2]}"] \& \& \& \\
\& v_3(2,2,-1,1)\ar[d,red]\ar[rd,brown] \& v_{15}(1,-1,0,1)\ar[rd,brown]\ar[ld,blue] \& v_{14}(2,1,0,-1)\ar[d,red]\ar[ld,blue,"{[2]}",pos=0.75] \& \& \\
\& v_5(1,0,0,1)\ar[d,red,"{[2]}",pos=0.25]\ar[rd,brown] \& v_4(2,2,0,-1)\ar[d,red] \& v_{16}(1,-1,1,-1)\ar[rd,mygreen]\ar[ld,blue] \& \& \\
\& v_8(0,-2,1,1)\ar[d,mygreen]\ar[rd,brown] \& v_6(1,0,1,-1)\ar[d,red,"{[2]}",pos=0.25]\ar[rd,mygreen] \& \& v_{17}(1,1,-1,0)\ar[d,red]\ar[ld,blue] \& \\
\& v_{10}(0,0,-1,2)\ar[rrrd,brown] \& v_9(0,-2,2,-1)\ar[rd,mygreen] \& v_7(1,2,-1,0)\ar[ld,red] \& v_{18}(0,-1,0,0)\ar[llld,blue] \& \\
\& v_{19}(0,0,0,0)\ar[rd,red] \& v_{20}(0,0,0,0)\ar[ld,blue]\ar[d,red,"{[2]}",pos=0.25]\ar[rd,mygreen,"2",pos=0.2] \& v_{21}(0,0,0,0)\ar[ld,red]\ar[d,mygreen,"{[2]_2}",pos=0.25]\ar[rd,brown] \& v_{22}(0,0,0,0)\ar[ld,mygreen]\ar[d,brown,"{[2]_2}",pos=0.25] \& \\
\& v_{\overline{18}}(0,1,0,0)\ar[d,red] \& v_{\overline{7}}(-1,-2,1,0)\ar[rd,mygreen]\ar[ld,blue] \& v_{\overline{9}}(0,2,-2,1)\ar[d,red]\ar[rd,brown] \& v_{\overline{10}}(0,0,1,-2)\ar[d,mygreen] \& \\
\& v_{\overline{17}}(-1,-1,1,0)\ar[rd,mygreen] \& \& v_{\overline{6}}(-1,0,-1,1)\ar[ld,blue]\ar[d,red,"{[2]}",pos=0.25]\ar[rd,brown] \& v_{\overline{8}}(0,2,-1,-1)\ar[d,red] \& \\
\& \& v_{\overline{16}}(-1,1,-1,1)\ar[d,red]\ar[rd,brown] \& v_{\overline{4}}(-2,-2,0,1)\ar[ld,blue]\ar[rd,brown] \& v_{\overline{5}}(-1,0,0,-1) \ar[ld,blue]\ar[d,red,"{[2]}",pos=0.25] \& \\
\& \& v_{\overline{14}}(-2,-1,0,1)\ar[rd,brown] \& v_{\overline{15}}(-1,1,0,-1)\ar[d,red] \& v_{\overline{3}}(-2,-2,1,-1)\ar[ld,blue]\ar[rd,mygreen] \& \\
\& \& \& v_{\overline{13}}(-2,-1,1,-1)\ar[rd,mygreen] \& \& v_{\overline{2}}(-2,0,-1,0)\ar[ld,blue] \\
\& \& \& \& v_{\overline{12}}(-2,1,-1,0)\ar[d,red] \& \\
\& \& \& \& v_{\overline{11}}(-3,-1,0,0)\ar[ld,blue] \& \\
\& \& \& v_{\overline{1}}(-3,0,0,0) \& \& 
}
\caption{The $40$-dimensional quantized adjoint module for $U_q$F$_{3\vert 1}$}
    \label{fig:F4 e-f action}
\end{figure}

\subsection{Quantum affine algebra and \texorpdfstring{$q$}--characters}

We consider the distinguished Dynkin diagram where one node is fermionic and the affine node attaches to this fermionic node, as shown below:

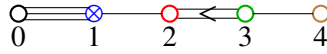
\begin{center}
\begin{tikzpicture}[root/.style={circle, draw, thick, minimum size=6pt, inner sep=0pt}]
    \node[root] at (0,0) {};
    \node at (1,0) {$\textcolor{blue}{\otimes}$};
    \node[root, color=red] at (2,0) {};
    \node[root, color=mygreen] at (3,0) {};
    \node[root, color=brown] at (4,0) {};

    \draw (0.1,0) -- (0.9,0);
    \draw (0.05,0.08) -- (0.95,0.08);
    \draw (0.05,-0.08) -- (0.95,-0.08);
    \draw (1.1,0) -- (1.9,0);
    \node at (2.5,0) {\scalebox{1.1}{$<$}};
    \draw (2.9,0.06)--(2.1,0.06);
    \draw (2.9,-0.06)--(2.1,-0.06);
    \draw (3.1,0) -- (3.9,0);

    \node at (0,-0.3) {0};
    \node at (1,-0.3) {1};
    \node at (2,-0.3) {2};
    \node at (3,-0.3) {3};
    \node at (4,-0.3) {4};
\end{tikzpicture}\ .
\end{center}
The corresponding Cartan matrix is given by 
$$\begin{bmatrix}
    2 & -1 & 0 & 0 & 0 \\
    -3 & 0 & 1 & 0 & 0 \\
    0 & -1 & 2 & -2 & 0 \\
    0 & 0 & -1 & 2 & -1 \\
    0 & 0 & 0 & -1 & 2
\end{bmatrix}\ .$$
We choose $d_0=-3, d_1=-1, d_2=1, d_3=d_4=2$ in this case.

The affine roots,  see \eqref{affine roots}, are given by
\bee{
A_{1,a}=\mr{2}_a^{-1}\ ,\quad A_{2,a}=\mr{1}_{q^{-1}a}^{qa}\mr{2}_{qa}\mr{2}_{q^{-1}a}\mr{3}_a^{-1}\ ,\quad A_{3,a}=\mr{2}_{qa}^{-1}\mr{2}_{q^{-1}a}^{-1}\mr{3}_{q^2a}\mr{3}_{q^{-2}a}\mr{4}_a^{-1}\ ,\quad A_{4,a}=\mr{3}_a^{-1}\mr{4}_{q^2a}\mr{4}_{q^{-2}a}\ .
}
Here $\mr{1}_b^a$, $\mr{2}_a$, $\mr{3}_a$, $\mr{4}_a$, $a,b\in\C^\times$, are quadruples of rational functions given by
\bee{
\mr{1}_b^a=\bigg(\sqrt{a/b}\frac{1-z/a}{1-z/b},1,1,1\bigg)\ ,\quad \mr{2}_a=\bigg(1,q^{-1}\frac{1-qz/a}{1-q^{-1}z/a},1,1\bigg)\ ,\\ 
\mr{3}_a=\bigg(1,1,q^{-2}\frac{1-q^{2}z/a}{1-q^{-2}z/a},1\bigg)\ ,\quad \mr{4}_a=\bigg(1,1,1,q^{-2}\frac{1-q^{2}z/a}{1-q^{-2}z/a}\bigg)\ .
}

We construct a $41$-dimensional $U_q\tl{\text{F}}_{3\vert 1}$-module  $\tl{L}_{(3,0,0,0)}$.
The diagram in Figure \ref{fig:F4 affine e-f action} shows the $q$-character of $\tl{L}_{(3,0,0,0)}$. Here we denote $\mr{1}_{q^b}^{q^a}$ by $\mr{1}_b^a$ and $\mr{i}_{q^a}$ by $\mr{i}_a$ for $i=2,3,4$. We omit writing the shift of the affine roots on the corresponding arrows, but it is clear from the corresponding $\ell$-weights. The supplement for this module is given as follows. The constants of the action of $X_i^-(z)$ are shown in Figure \ref{fig:F4 affine e-f action}. Here all the numbers on arrows are quantum numbers, and if a number is not shown on an arrow then it is one. As before, we have $\ol{i}=41-i$. The constants of action of $X_i^+(z)$ can be obtained from those of $X_i^-(z)$ shown in the diagram in Figure \ref{fig:F4 affine e-f action}. Specifically,  when $1\le j,k\le 18$, we have $X_i^+(z)\,\tl v_j=a\,\delta(z/q^s)\,\tl v_k$ if and only if $X_i^-(z)\,\tl v_{\overline{k}}=a\,\delta(z/q^s)\,\tl v_{\overline{j}}$, and when $23\le j,k\le 40$ we have $X_i^+(z)\,\tl v_j=(-1)^{s_i}a\,\D(z/q^s)\,\tl v_k$ if and only if $X_i^-(z)\,\tl v_{\overline{k}}=a\,\D(z/q^s)\,\tl v_{\overline{j}}$. In the case when $j=19,20,21,22,41$ and $k=7,9,10,18$, we have $X_i^+(z)\,\tl v_j=a\,\D(z/q^s)\,\tl v_k$ if and only if $X_i^-(z)\,\tl v_j=a\,\D(z/q^s)\,\tl v_{\overline{k}}$, and $X_i^+(z)\,\tl v_{\ol{k}}=a\,\D(z/q^s)\,\tl v_{j}$ if and only if $X_i^-(z)\,\tl v_k=a\,\D(z/q^s)\,\tl v_j$.
\begin{figure}[ht!]
    \centering
\ddia{row sep=10pt, column sep=-1pt}{
\& \& \tl{v}_1\big(\mr{1}_0^6\big) \ar[ld,blue,"{[3]}"] \& \& \& \\
\& \tl{v}_{11}\big(\mr{1}_0^6\,\mr{2}_0\big) \ar[d,red] \& \& \& \& \\
\&  \tl{v}_{12}\big(\mr{1}_2^6\,\mr{2}_2^{-1}\mr{3}_1\big) \ar[rd,mygreen]\ar[ld,blue,"{[2]}"] \& \& \& \& \\
\tl{v}_2\big(\mr{1}_2^6\,\mr{3}_1\big)\ar[rd,mygreen] \& \& \tl{v}_{13}\big(\mr{1}_2^6\,\mr{2}_4\mr{3}_5^{-1}\mr{4}_3\big) \ar[d,red]\ar[rd,brown]\ar[ld,blue,"{[2]}"] \& \& \& \\
\& \tl{v}_3\big(\mr{1}_2^6\,\mr{2}_2\mr{2}_4\mr{3}_5^{-1}\mr{4}_3\big) \ar[d,red]\ar[rd,brown] \& \tl{v}_{15}\big(\mr{1}_2^4\,\mr{2}_6^{-1}\mr{4}_3\big) \ar[rd,brown]\ar[ld,blue] \& \tl{v}_{14}\big(\mr{1}_2^6\,\mr{2}_4\mr{4}_7^{-1}\big) \ar[d,red]\ar[ld,blue,"{[2]}",pos=0.75] \& \& \\
\& \tl{v}_5\big(\mr{1}_2^4\,\mr{2}_2\mr{2}_6^{-1}\mr{4}_3\big) \ar[d,red,"{[2]}",pos=0.15]\ar[rd,brown] \& \tl{v}_4\big(\mr{1}_2^6\,\mr{2}_2\mr{2}_4\mr{4}_7^{-1}\big) \ar[d,red] \& \tl{v}_{16}\big(\mr{1}_2^4\,\mr{2}_6^{-1}\mr{3}_5\mr{4}_7^{-1}\big) \ar[rd,mygreen]\ar[ld,blue] \& \& \\
\& \tl{v}_8\big(\mr{2}_4^{-1}\mr{2}_6^{-1}\mr{3}_3\mr{4}_3\big) \ar[d,mygreen]\ar[rd,brown] \& \tl{v}_6\big(\mr{1}_2^4\,\mr{2}_2\mr{2}_6^{-1}\mr{3}_5\mr{4}_7^{-1}\big) \ar[d,red,"{[2]}",pos=0.15]\ar[rd,mygreen] \& \& \tl{v}_{17}\big(\mr{1}_2^4\,\mr{2}_8\mr{3}_9^{-1}\big) \ar[d,red]\ar[ld,blue] \& \\
\& \tl{v}_{10}\big(\mr{3}_7^{-1}\mr{4}_3\mr{4}_5\big) \ar[rrrd,brown,"-1"',pos=0.2]\ar[rrrrd,brown] \& \tl{v}_9\big(\mr{2}_4^{-1}\mr{2}_6^{-1}\mr{3}_3\mr{3}_5\mr{4}_7^{-1}\big) \ar[rd,mygreen,"{[3]}"',pos=0.85]\ar[rrd,mygreen] \& \tl{v}_7\big(\mr{1}_2^4\,\mr{2}_2\mr{2}_8\mr{3}_9^{-1}\big) \ar[ld,red,"{[4]}"',pos=0.15]\ar[d,red,"{[2]}",pos=0.1] \& \tl{v}_{18}\big(\mr{1}_{2,10}^{4,8}\,\mr{2}_{10}^{-1}\big) \ar[llld,blue,"{-1/[4]}"',pos=0.8]\ar[lld,blue,"{[3]}",pos=0.2] \& \\[1cm]
\& \tl{v}_{19}\big(\mr{1}_{2,10}^{4,8}\big) \ar[d,blue,"{[3]}"',pos=0.25] \& \tl{v}_{20}\big(\mr{1}_{2,10}^{4,8}\,\mr{2}_2\mr{2}_{10}^{-1}\big)\ar[ld,blue,"{1/[4]}"']\ar[d,red,"{1/[3]}",pos=0.25] \& \tl{v}_{21}\big(\mr{2}_8\mr{2}_4^{-1}\mr{3}_3\mr{3}_9^{-1}\big)\ar[ld,red,"{1/[3]}",pos=0.2]\ar[d,mygreen] \& \tl{v}_{22}\big(\mr{3}_5\mr{3}_7^{-1}\mr{4}_5\mr{4}_7^{-1}\big) \ar[ld,mygreen,"-1"]\ar[d,brown] \& \tl{v}_{41}\big(\mr{4}_3\mr{4}_9^{-1}\big)\ar[ld,brown,"{[3]}"] \\
\& \tl{v}_{\ol{18}}\big(\mr{1}_{2,10}^{4,8}\,\mr{2}_2\big) \ar[d,red] \& \tl{v}_{\ol{7}}\big(\mr{1}_{10}^8\,\mr{2}_4^{-1}\mr{2}_{10}^{-1}\mr{3}_3\big) \ar[rd,mygreen]\ar[ld,blue] \& \tl{v}_{\ol{9}}\big(\mr{2}_6\mr{2}_8\mr{3}_7^{-1}\mr{3}_9^{-1}\mr{4}_5\big) \ar[d,red]\ar[rd,brown] \& \tl{v}_{\ol{10}}\big(\mr{3}_5\mr{4}_7^{-1}\mr{4}_9^{-1}\big) \ar[d,mygreen] \& \\
\& \tl{v}_{\ol{17}}\big(\mr{1}_{10}^8\,\mr{2}_4^{-1}\mr{3}_3\big) \ar[rd,mygreen] \& \& \tl{v}_{\ol{6}}\big(\mr{1}_{10}^8\,\mr{2}_6\mr{2}_{10}^{-1}\mr{3}_7^{-1}\mr{4}_5\big) \ar[ld,blue]\ar[d,red,"{[2]}",pos=0.15]\ar[rd,brown] \& \tl{v}_{\ol{8}}\big(\mr{2}_6\mr{2}_8\mr{3}_9^{-1}\mr{4}_9^{-1}\big) \ar[d,red] \& \\
\& \& \tl{v}_{\ol{16}}\big(\mr{1}_{10}^8\,\mr{2}_6\mr{3}_7^{-1}\mr{4}_5\big) \ar[d,red]\ar[rd,brown] \& \tl{v}_{\ol{4}}\big(\mr{1}_{10}^6\,\mr{2}_8^{-1}\mr{2}_{10}^{-1}\mr{4}_5\big) \ar[ld,blue]\ar[rd,brown] \& \tl{v}_{\ol{5}}\ar[ld,blue]\ar[d,red,"{[2]}",pos=0.15]\big(\mr{1}_{10}^8\,\mr{2}_6\mr{2}_{10}^{-1}\mr{4}_9^{-1}\big) \& \\
\& \& \tl{v}_{\ol{14}}\big(\mr{1}_{10}^6\,\mr{2}_8^{-1}\mr{4}_5\big) \ar[rd,brown] \& \tl{v}_{\ol{15}}\big(\mr{1}_{10}^8\,\mr{2}_6\mr{4}_9^{-1}\big) \ar[d,red] \& \tl{v}_{\ol{3}}\big(\mr{1}_{10}^6\,\mr{2}_8^{-1}\mr{2}_{10}^{-1}\mr{3}_7\mr{4}_9^{-1}\big) \ar[ld,blue]\ar[rd,mygreen] \& \\
\& \& \& \tl{v}_{\ol{13}}\big(\mr{1}_{10}^6\,\mr{2}_8^{-1}\mr{3}_7\mr{4}_9^{-1}\big) \ar[rd,mygreen] \& \& \tl{v}_{\ol{2}}\big(\mr{1}_{10}^6\,\mr{3}_{11}^{-1}\big) \ar[ld,blue] \\
\& \& \& \& \tl{v}_{\ol{12}}\big(\mr{1}_{10}^6\,\mr{2}_{10}\mr{3}_{11}^{-1}\big) \ar[d,red] \& \\
\& \& \& \& \tl{v}_{\ol{11}}\big(\mr{1}_{12}^6\,\mr{2}_{12}^{-1}\big) \ar[ld,blue] \& \\
\& \& \& \tl{v}_{\ol{1}}\big(\mr{1}_{12}^6\big) \& \& 
}
\caption{The $q$-character and the supplement for the smallest non-trivial module for $U_q\tl{\text{F}}_{3|1}$}
    \label{fig:F4 affine e-f action}
\end{figure}

For all $z\in C^\times$, the restriction of $\tl{L}_{(3,0,0,0)}(z)$ to $U_q$F$_{3\vert 1}$ decomposes as 
$$\tl{L}_{(3,0,0,0)}(z)\cong L_{(3,0,0,0)}\oplus L_{(0,0,0,0)}\ .$$

We now describe the action of $U_q\tl{\text{F}}_{3\vert 1}$ on $\tl{L}_{(3,0,0,0)}(z)$ in the Drinfeld-Jimbo realization. We choose a basis $\{v_i\}_{i=1}^{40}\cup\{v_{41}\}$ for $\tl{L}_{(3,0,0,0)}(z)$, where $\{v_i\}_{i=1}^{40}$ is the basis of the $U_q$F$_{3\vert 1}$-module $L_{(3,0,0,0)}$ as  in Figure \ref{fig:F4 e-f action} and $v_{41}$ is a basis of the $1$-dimensional $U_q$F$_{3\vert 1}$-module $L_{(0,0,0,0)}$. Note that $v_i$ are scalar multiples of $\tl{v}_i$ for all $i$, except for $i=19,20,21,22,41$.

Let
\eq{\label{superF:E0 action}
E_0=z\bigg( & -\frac{[2]_5}{[3]^{\mr{i}}}E_{19,1} + [3]\,E_{20,1} - \frac{[4]}{[3]^{\mr{i}}} E_{21,1} + \frac{[2]}{[3]^{\mr{i}}} E_{22,1} + \frac{[2]}{[3]^{\mr{i}}} E_{41,1} + E_{\overline{1},19} + \big([2]^{\mr{i}}\big)^2 E_{\overline{1},41} +\sum_{i=11}^{18}E_{12+i,i}\bigg)\ ,
}
\eq{\label{superF: F0 action}
F_0=z^{-1}\bigg( & \frac{[2]_5}{[3]^{\mr{i}}} E_{19,\overline{1}} - [3]\, E_{20,\overline{1}} + \frac{[4]}{[3]^{\mr{i}}} E_{21,\overline{1}} - \frac{[2]}{[3]^{\mr{i}}} E_{22,\overline{1}} - \frac{[2]}{[3]^{\mr{i}}} E_{41,\overline{1}} - E_{1,19} -\big([2]^{\mr{i}}\big)^2 E_{1,41} + \sum_{i=11}^{18}E_{i,12+i} \bigg)\ .
}
where $E_{ij}$ are matrix units, that is, $E_{ij}v_k=\D_{jk}v_i$.

\begin{prp}
    The action of $E_i, F_i$, $i=1,2,3,4$ as in Figure  \ref{fig:F4 e-f action} and $E_0, F_0$ as in \eqref{superF:E0 action}, \eqref{superF: F0 action}  gives a $U_q\tl{\text{F}}_{3|1}$-module structure.
\end{prp}
\begin{proof}
    We check all the relations in $U_q\tl{\text{F}}_{3|1}$ by a direct computation.
\end{proof}

\subsection{The \texorpdfstring{$R$}--Matrix}
The next theorem describes the tensor square of
$\tl{L}_{(3,0,0,0)}$. This module is completely reducible as a $U_q$F$_{3\vert 1}$-module.
\begin{thm}
As $U_q$F$_{3\vert 1}$-modules, we have 
\eq{\label{superF:tensor adjoint}
\big(\tl{L}_{(3,0,0,0)}\big)^{\otimes 2} \ & \cong\big(\underbracket[0.1ex]{L_{(3,0,0,0)}}_{40}\oplus \underbracket[0.1ex]{L_{(0,0,0,0)}}_{1}\big)^{\otimes 2} \\
& \cong \underbracket[0.1ex]{L_{(6,0,0,0)}}_{296}\oplus \underbracket[0.1ex]{L_{(6,1,0,0)}}_{756}\oplus \underbracket[0.1ex]{L_{(5,0,1,0)}}_{507}\oplus\,  3\underbracket[0.1ex]{L_{(3,0,0,0)}}_{40}\oplus \, 2\underbracket[0.1ex]{L_{(0,0,0,0)}}_{1}\ .
}
\end{thm}
\p{We again perform the direct calculations in $\tl{L}_{(3,0,0,0)}^{\otimes 2}$ and observe the following.

There are $8$ linearly independent singular vectors, which we enumerate as $u_1,u_2,u_3,u_{4a},u_{4b},u_{4c},u_{5a},u_{5b}$. The singular vector $u_1=v_1\otimes v_1$ generates $L_{(6,0,0,0)}$, the singular vector $u_2=q^{-3}v_1\otimes v_{11}-q^3 v_{11}\otimes v_1$ generates $L_{(6,1,0,0)}$, and the singular vector $u_3=q^{-4}v_1\otimes v_2+q^4 v_2\otimes v_1-q v_{11}\otimes v_{12}+q^{-1}v_{12}\otimes v_{11}$ generates $L_{(5,0,1,0)}$. There is a natural choice for the three singular vectors $u_{4a}$, $u_{4b}$, $u_{4c}$ of weight $(3,0,0,0)$. We choose $u_{4b}=v_1\otimes v_{41}$, $u_{4c}=v_{41}\otimes v_1$, and $u_{4a}\in L_{(3,0,0,0)}^{\otimes 2}$ is given by 
\bee{
u_{4a}=\ & [2]_5\,(q^3v_1\otimes v_{19}-q^{-3}v_{19}\otimes v_1)-[3]_2\,(q^3v_1\otimes v_{20}-q^{-3}v_{20}\otimes v_1)+[4]\,(q^3v_1\otimes v_{21}-q^{-3}v_{21}\otimes v_1) \\
& -[2]\,(q^3v_1\otimes v_{22}-q^{-3}v_{22}\otimes v_1)+[3]_2\,(q^{9/2}v_{11}\otimes v_{18}+q^{-9/2}v_{18}\otimes v_{11} -q^{7/2}v_{12}\otimes v_{17}-q^{-7/2}v_{17}\otimes v_{12} \\
& +q^{3/2}v_{13}\otimes v_{16}+q^{-3/2}v_{16}\otimes v_{13} -q^{-1/2}v_{14}\otimes v_{15}-q^{1/2}v_{15}\otimes v_{14})\ .
}
Finally, the two singular vectors of weight zero are chosen so that $u_{5b}=v_{41}\otimes v_{41}$, and $u_{5a}\in L_{(3,0,0,0)}^{\otimes 2}$ is given by
\bee{
u_{5a}=\ & [3]^{\mr{i}}\,(q^9 v_1\otimes v_{\ol{1}}+q^{-9} v_{\ol{1}}\otimes v_1) +[2][3]_2\,(q^8 v_2\otimes v_{\ol{2}}+q^{-8} v_{\ol{2}}\otimes v_2-q^6v_3\otimes v_{\ol{3}}-q^{-6}v_{\ol{3}}\otimes v_3 +q^4v_4\otimes v_{\ol{4}} \\ 
& +q^{-4}v_{\ol{4}}\otimes v_4-q^4v_8\otimes v_{\ol{8}}-q^{-4}v_{\ol{8}}\otimes v_8 + q^2 v_9\otimes v_{\ol{9}}+q^{-2}v_{\ol{9}}\otimes v_9 + q^2v_{10}\otimes v_{\ol{10}}+q^{-2}v_{\ol{10}}\otimes v_{10} ) \\
& + [3]_2\,(q^5v_5\otimes v_{\ol{5}}+q^{-5}v_{\ol{5}}\otimes v_5 -q^3v_6\otimes v_{\ol{6}}-q^{-3}v_{\ol{6}}\otimes v_6 +qv_7\otimes v_{\ol{7}}+q^{-1}v_{\ol{7}}\otimes v_7 + q^9v_{11}\otimes v_{\ol{11}} \\ 
& -q^{-9}v_{\ol{11}}\otimes v_{11} -q^8v_{12}\otimes v_{\ol{12}} +q^{-8}v_{\ol{12}}\otimes v_{12} +q^6v_{13}\otimes v_{\ol{13}} -q^{-6}v_{\ol{13}}\otimes v_{13} -q^4v_{14}\otimes v_{\ol{14}} +q^{-4}v_{\ol{14}}\otimes v_{14} \\ 
& - q^5v_{15}\otimes v_{\ol{15}} +q^{-5}v_{\ol{15}}\otimes v_{15} +q^3v_{16}\otimes v_{\ol{16}} -q^{-3}v_{\ol{16}}\otimes v_{16} -qv_{17}\otimes v_{\ol{17}} +q^{-1}v_{\ol{17}}\otimes v_{17} +v_{18}\otimes v_{\ol{18}} \\
& -v_{\ol{18}}\otimes v_{18}-v_{19}\otimes v_{20}-v_{20}\otimes v_{19}) +[2]_5\,v_{19}\otimes v_{19} + [4]\,(v_{19}\otimes v_{21}+v_{21}\otimes v_{19}-v_{21}\otimes v_{21}-v_{22}\otimes v_{22}) \\
& -[2]\,(v_{19}\otimes v_{22}+v_{22}\otimes v_{19}-v_{21}\otimes v_{22}-v_{22}\otimes v_{21})\ .
}
}

For $\la=(6,0,0,0),(6,1,0,0),(5,0,1,0),(3,0,0,0),(0,0,0,0)$, let $P_\la^q$ be the projector onto the summand $L_{\la}$ in the decomposition \eqref{superF:tensor adjoint}.

\begin{thm}
\label{superF:Rcheck}
    In terms of projectors, we have
\eq{\label{superF: Rqz}
\check{R}(z)=\ & P^q_{(6,0,0,0)}-q^6\frac{1-q^{-6}z}{1-q^6 z}\,P^q_{(6,1,0,0)}+q^{10}\frac{(1-q^{-4}z)(1-q^{-6}z)}{(1-q^4 z)(1-q^6 z)}\,P^q_{(5,0,1,0)} \\
& + \frac{q^3\ f_1(z)}{(1-q^{-4}z)(1-q^4 z)(1-q^6 z)}\otimes P^q_{(3,0,0,0)} + \frac{f_0(z)}{(1-q^{-4}z)(1-q^{-6}z)(1-q^4 z)(1-q^6 z)}\otimes P^q_{(0,0,0,0)}\ ,
}
where the matrices $f_1(z)$ and $f_0(z)$ are given by 
\bee{
& f_1(z)=\begin{bmatrix}
    -q^{-3}+\A_q\,z-\A_{q^{-1}}\,z^2+q^3z^3 & \hspace{-20pt} -\B\,z(1-z) & \hspace{-10pt} -\B\,z(1-z) 
    \vspace{5pt} \\
    -\G\,z(1-z) & \hspace{-20pt} \mu\,z(q^3+q^{-3}z) & \hspace{-10pt} (1-z)(q^3-\nu\,z+q^{-3}z^2) 
    \vspace{5pt} \\
    -\G\,z(1-z) & \hspace{-20pt} (1-z)(q^3-\nu\,z+q^{-3}z^2) & \hspace{-10pt} \mu\,z(q^3+q^{-3}z)
\end{bmatrix} \ ,
\vspace{10pt} \\
\vspace{10pt} \\
& f_0(z)= \begin{bmatrix}
    q^{-6}-q^{-3}\zeta\,z+\xi\,z^2-q^3\zeta\,z^3+q^6z^4 & \eta\, z(1-z^2) 
    \vspace{5pt} \\
    \rho\, z(1-z^2) & q^{6}-q^{3}\zeta\,z+\xi\,z^2-q^{-3}\zeta\,z^3+q^{-6}z^4
\end{bmatrix} \ .
}

\bigskip

Here, the constants $\A_q,\B,\G,\mu,\nu,\zeta,\xi,\eta,\rho\in\C(q)$ are given by 
\bee{
& \A_q=\frac{(q^2[2]-4q^{-1}+2q^{-4}[2]-q^{-7})\,[3]}{[3]^{\mr{i}}}\ ,\quad \B=\frac{([2]^{\mr{i}})^3\,[3]}{[3]^{\mr{i}}}\ ,
\vspace{8pt} \\
& \G=\frac{2\,[2]^{\mr{i}}\,[5]^\mr{i}\,([2])^2\,[3]}{[3]^{\mr{i}}}\ ,\quad \mu=\frac{([2]^{\mr{i}})^3\,[2]\,[3]}{[3]^{\mr{i}}}\ ,\quad \nu=\frac{[2]\,[3]_2^{\mr{i}}}{[3]^{\mr{i}}}\ ,\quad \zeta= \frac{2[2]_5}{[3]^{\mr{i}}}\ , 
\vspace{8pt} \\
& \xi=[3]\,\Big([2]_8-[2]_6-[2]_4+2[2]_2\Big)\ ,\quad \eta=\frac{([2]^{\mr{i}})^5\,[3]}{[3]^{\mr{i}}}\ ,\quad \rho=\frac{[2]^{\mr{i}}\,[2]_4\,[3]\,([4])^2}{[3]^{\mr{i}}}\ . 
}
\p{
We use the  equation $$\check{R}(z)\,\Delta(F_0)=\Delta(F_0)\,\check{R}(z):\tl L_{(3,0,0,0)}(z)\otimes\tl{L}_{(3,0,0,0)}(1)\to \tl L_{(3,0,0,0)}(1)\otimes\tl{L}_{(3,0,0,0)}(z)\ ,$$ with $F_0$ in \eqref{superF: F0 action}. Equivalently, one can use $E_0$ in \eqref{superF:E0 action}.
}
\end{thm}

\begin{rmk}
    The new feature here is that we have poles of the $R$-matrix simultaneously, at $z=q^{-4}$, $z=q^4$ and at $z=q^{-6}$, $z=q^6$. We also note that the determinants of 
    $$g_1(z)=\frac{q^3\ f_1(z)}{(1-q^{-4}z)(1-q^4 z)(1-q^6 z)}\ ,\quad g_2(z)=\frac{f_0(z)}{(1-q^{-4}z)(1-q^{-6}z)(1-q^4 z)(1-q^6 z)}\ ,$$
     are given by 
     $$\det\big(g_1(z)\big)=q^{12}\frac{(1-q^{-6}z)^2}{(1-q^6z)^2}\ ,\quad \det\big(g_2(z)\big)=1\ .$$
\end{rmk}

The tensor product $\tl{L}_{(3,0,0,0)}(z)\otimes\tl{L}_{(3,0,0,0)}(1)$ is irreducible except for some special values of $z$. These special values of $z$, and the corresponding non-trivial submodules, are listed below. Note that the submodules for $z=q^k$, are quotient modules for $z=q^{-k}$.
\begin{enumerate}
    \item $z=q^4$: We have a $548$-dimensional submodule isomorphic to $L_{(5,0,1,0)}\oplus L_{(3,0,0,0)}\oplus L_{(0,0,0,0)}$ as a $U_q$F$_{3\vert 1}$-module, and generated by $u_3$. This submodule is contained inside a $1640$-dimensional submodule isomorphic to $L_{(6,0,0,0)}\oplus L_{(6,1,0,0)}\oplus L_{(5,0,1,0)}\oplus 2\,L_{(3,0,0,0)}\oplus L_{(0,0,0,0)}$ as a $U_q$F$_{3\vert 1}$-module, and generated by $u_1$.
    \item $z=q^{-4}$: We have a $41$-dimensional submodule isomorphic to $L_{(3,0,0,0)}\oplus L_{(0,0,0,0)}$ as a $U_q$F$_{3\vert 1}$-module, and generated by $u_{4a}+u_{4b}+u_{4c}$. This submodule is contained inside a $1133$-dimensional submodule isomorphic to $L_{(6,0,0,0)}\oplus L_{(6,1,0,0)}\oplus 2\,L_{(3,0,0,0)}\oplus L_{(0,0,0,0)}$ as a $U_q$F$_{3\vert 1}$-module, and generated by $u_1$.
    \item $z=q^{6}$: We have a $1344$-dimensional submodule isomorphic to $L_{(6,1,0,0)}\oplus L_{(5,0,1,0)}\oplus 2\,L_{(3,0,0,0)}\oplus L_{(0,0,0,0)}$ as a $U_q$F$_{3\vert 1}$-module, and generated by $u_2$. This submodule is contained inside a $1680$-dimensional submodule isomorphic to $L_{(6,0,0,0)}\oplus L_{(6,1,0,0)}\oplus L_{(5,0,1,0)}\oplus 3\,L_{(3,0,0,0)}\oplus L_{(0,0,0,0)}$ as a $U_q$F$_{3\vert 1}$-module, and generated by $u_1$.
    \item $z=q^{-6}$: We have a $1$-dimensional submodule isomorphic to $L_{(0,0,0,0)}$ as a $U_q$F$_{3\vert 1}$-module, and generated by $[2]_2\,u_{5a}-u_{5b}$. This submodule is contained inside a $337$-dimensional submodule isomorphic to $L_{(6,0,0,0)}\oplus L_{(3,0,0,0)}\oplus L_{(0,0,0,0)}$ as a $U_q$F$_{3\vert 1}$-module, and generated by $u_1$.
\end{enumerate}

We consider the $q\to 1$ rational limit. With our normalization of the singular vectors $u_1,\dots, u_{5b}$, the limit of these vectors exists, and therefore the limit of the projectors exists. Let $P_{(a,b,c,d)}$ be the limit $q\to 1$ of the $U_q$F$_{3\vert 1}$-projectors $P_{(a,b,c,d)}^q$ appearing in the expression \eqref{superF: Rqz} of $\check{R}(z)$.

\begin{cor}
    The rational $R$-matrix $\check{R}(u)$ for the Yangian of F$_{3\vert 1}$ is given by 
    \eq{\label{superF: rational R matrix}
\check{R}(u) =\ & P_{(6,0,0,0)} + \frac{3-u}{3+u}\,P_{(6,1,0,0)} + \frac{(2-u)(3-u)}{(2+u)(3+u)}\,P_{(5,0,1,0)} + \frac{\bar f_1(u)}{(2-u)(2+u)(3+u)}\otimes P_{(3,0,0,0)} \\
& + \frac{\bar f_0(u)}{(2-u)(3-u)(2+u)(3+u)}\otimes P_{(0,0,0,0)}\ ,
}
where the matrices $\bar f_1(u)$ and $\bar f_0(u)$ are given by 
\bee{
& \bar f_1(u)=\begin{bmatrix}
    12-4u+3u^2+u^3 &  6\,u &  6\,u 
    \vspace{5pt} \\
    12\,u & 12 & u(4-u)(1+u) 
    \vspace{5pt} \\
    12\,u & u(4-u)(1+u) & 12
\end{bmatrix} \ ,
\vspace{10pt} \\
\vspace{10pt} \\
& \bar f_0(u)= \begin{bmatrix}
    36-12u+5u^2+6u^3+u^4 & -24\,u 
    \vspace{5pt} \\
    -48\,u & 36+12u+5u^2-6u^3+u^4
\end{bmatrix} \ .
}
\p{
We substitute $z=q^{2u}$ in \eqref{superF: Rqz} and take limit $q\to 1$. Apriori, the $q\to 1$ limit of $\check{R}(q^{2u})$ does not exist. To have a limit we conjugate $\check{R}(q^{2u})$ by $A\otimes A$ where the operator $A:V\to V$ is given by $A\,v_i=v_i$ for $1\le i\le 40$ and $A\,v_{41}=\frac{1}{[2]^{\mr{i}}}v_{41}$. Thus, we obtain $\check{R}(u)=\lim_{q\to 1}(A\otimes A)\check{R}(q^{2u})(A\otimes A)^{-1}$ as above in \eqref{superF: rational R matrix}. Note that conjugation by an operator of the form $A\otimes A$ preserves the quantum Yang-Baxter equation and, for our choice of  $A$, $A\otimes A$ commutes with action of  $U_q{\text{F}}_{3|1}$ but not with action of $U_q\tl{\text{F}}_{3|1}$.
}
\end{cor}

\section{Type G\texorpdfstring{$_{2\vert 1}$}{2}}
\label{super:G}

In this section, we give the expression of $\check{R}(z)$ for the $32$-dimensional representation of $U_q\tl{\text{G}}_{2\vert 1}$ in terms of projectors related to the tensor square of the $31$-dimensional representation of $U_q$G$_{2\vert 1}$. As in the case of $U_q\tl{\text{F}}_{3\vert 1}$, the $R$-matrix obtained has the form similar to the $R$-matrices of untwisted and twisted types in cases of non-trivial multiplicity.

\subsection{Quantum algebra of finite type}

We consider the distinguished Dynkin diagram where one node is fermionic and the other two bosonic, as shown below:

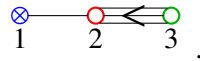
\begin{center}
\begin{tikzpicture}
    [root/.style={circle, draw, thick, minimum size=6pt, inner sep=0pt}]
    \node at (1,0) {$\textcolor{blue}{\otimes}$};
    \node[root, color=red] at (2,0) {};
    \node[root, color=mygreen] at (3,0) {};

    \node at (2.5,0) {\scalebox{1.5}{$<$}};
    \draw (2.9,0)--(2.1,0);
    \draw (2.95,0.1)--(2.05,0.1);
    \draw (2.95,-0.1)--(2.05,-0.1);
    \draw (1.1,0) -- (1.9,0);

    \node at (1,-0.3) {1};
    \node at (2,-0.3) {2};
    \node at (3,-0.3) {3};
\end{tikzpicture}\ .
\end{center}
The corresponding Cartan matrix is given by 
$$\begin{bmatrix}
    0 & 1 & 0 \\
    -1 & 2 & -3 \\
    0 & -1 & 2
\end{bmatrix}\ .$$
We choose $d_1=-1$, $d_2=1$, $d_3=3$ in this case.

The finite dimensional irreducible modules of G$_{2\vert 1}$ are classified (see \cite{K75}, \cite{M13} for details) by triples $(a,b,c)$ of non-negative integers such that $k=(a-2b-3c)/2$ is a non-negative integer different from $1$ and if $k=0$ then $(a,b,c)=(0,0,0)$, if  $k=2$ then $b=0$. 

 We denote the irreducible $U_q$G$_{2\vert 1}$-module corresponding to the highest weight $(a,b,c)$ by $L_{(a,b,c)}$.

We construct a weighted basis $\{v_i\}_{i=1}^{31}$ for the $31$-dimensional $U_q$G$_{2\vert 1}$-module $L_{(4,0,0)}$ as follows. The highest weight vector $v_1$ is of weight $(4,0,0)$. Other vectors $v_i$ are such that the action of $U_q$G$_{2\vert 1}$ is as in Figure \ref{fig:G3 e-f action}. Figure \ref{fig:G3 e-f action} shows the action $F_i$, $i=1,2,3$. We use blue for $F_1$, red for $F_2$, and green for $F_3$, as in the Dynkin diagram nodes. Here $\overline{i}=32-i$. The coefficients on the arrows are quantum numbers, and if a coefficient is not shown, then it is one. 
For example, $F_2\,v_{16}=[2]\,v_{\overline{6}}$, $F_3\,v_{17}=[2]_3\,v_{\overline{7}}$, etc. 
The action of $E_i$, $i=1,2,3$ is symmetric with respect to the action of $F_i$ and can be read off from the diagram obtained by reflecting the diagram shown in Figure \ref{fig:G3 e-f action} about a horizontal line passing through the weight zero vectors $v_{15},v_{16},v_{17}$. 
Specifically, when $1\le j,k\le 14$, we have $E_i\,v_j=a\,v_k$ if and only if $F_i\,v_{\overline{k}}=a\,v_{\overline{j}}$, and when $18\le j,k\le 31$ we have $E_i\,v_j=(-1)^{s_i}a\,v_k$ if and only if $F_i\,v_{\overline{k}}=a\,v_{\overline{j}}$. In the case when $15\le j\le 17$ and $k=6,7,14$, we have $E_i\,v_j=a\,v_k$ if and only if $F_i\,v_j=a\,v_{\overline{k}}$, and $E_i\,v_{\ol{k}}=(-1)^{s_i}a\,v_{j}$ if and only if $F_i\,v_k=a\,v_j$. The numbering of the vectors $v_1,\dots,v_{31}$ is chosen such that $v_{8},\dots,v_{14}$ and $v_{\overline{14}},\dots,v_{\overline{8}}$ are vectors of odd parity, and the rest of the vectors are of even parity. 

\begin{figure}[ht!]
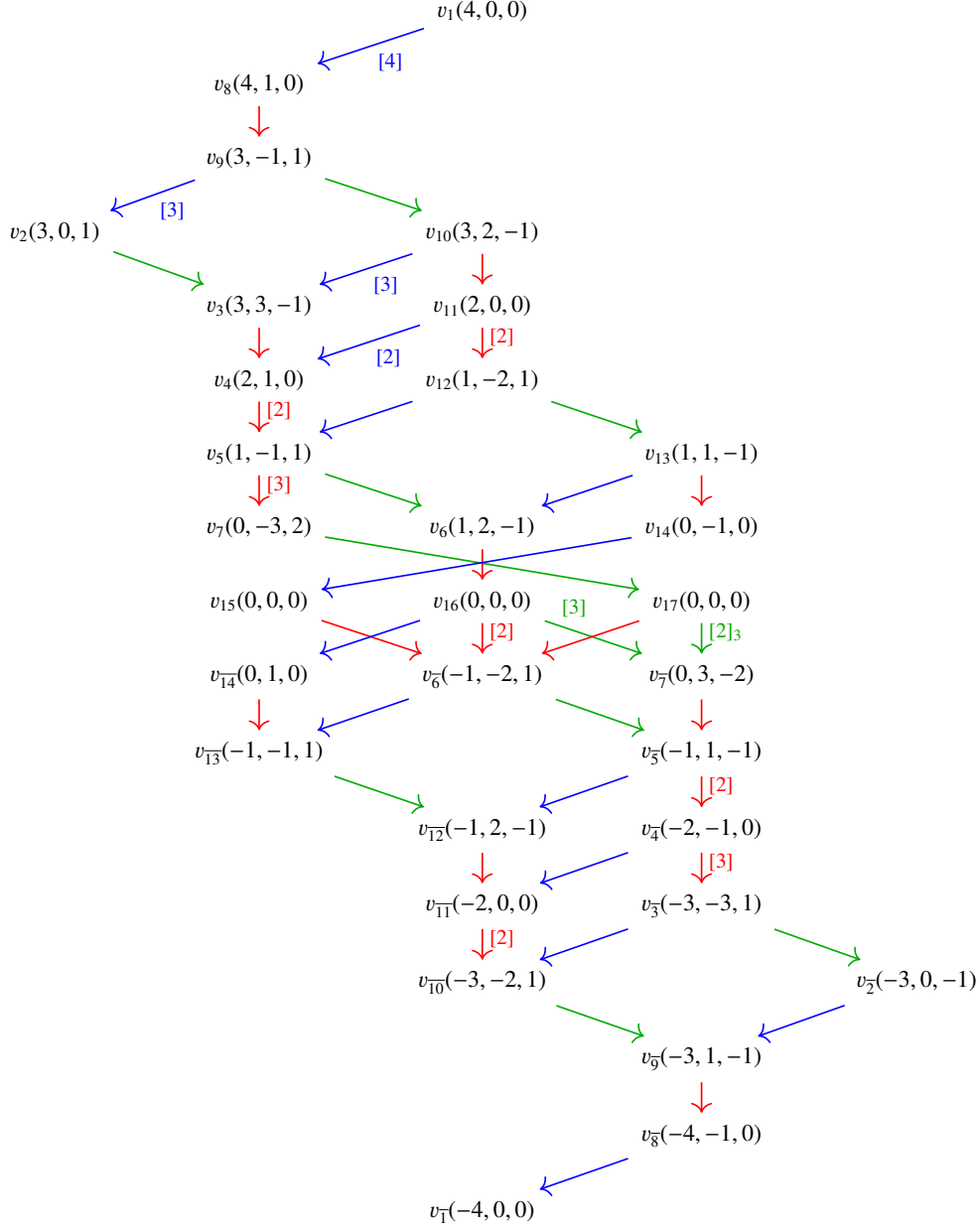

    \centering
\ddia{row sep=12pt, font=\footnotesize}{
\& \& v_1(4,0,0) \ar[ld,blue,"{[4]}"] \& \& \\
\& v_8(4,1,0) \ar[d,red] \& \& \& \\
\& v_9(3,-1,1) \ar[rd,mygreen]\ar[ld,blue,"{[3]}"] \& \& \& \\
v_2(3,0,1) \ar[rd,mygreen] \& \& v_{10}(3,2,-1) \ar[d,red]\ar[ld,blue,"{[3]}"] \& \& \\
\& v_3(3,3,-1) \ar[d,red] \& v_{11}(2,0,0)\ar[d,red,"{[2]}",pos=0.3]\ar[ld,blue,"{[2]}"] \& \& \\
\& v_4(2,1,0)\ar[d,red,"{[2]}",pos=0.3] \& v_{12}(1,-2,1)\ar[ld,blue]\ar[rd,mygreen] \& \& \\
\& v_5(1,-1,1)\ar[d,red,"{[3]}",pos=0.3]\ar[rd,mygreen] \& \& v_{13}(1,1,-1)\ar[ld,blue]\ar[d,red] \& \\
\& v_7(0,-3,2)\ar[rrd,mygreen] \& v_6(1,2,-1)\ar[d,red] \& v_{14}(0,-1,0)\ar[lld,blue] \& \\
\& v_{15}(0,0,0)\ar[rd,red] \& v_{16}(0,0,0)\ar[ld,blue]\ar[d,red,"{[2]}",pos=0.3]\ar[rd,mygreen,"{[3]}",pos=0.1] \& v_{17}(0,0,0)\ar[d,mygreen,"{[2]_3}",pos=0.3]\ar[ld,red] \& \\
\& v_{\overline{14}}(0,1,0)\ar[d,red] \& v_{\overline{6}}(-1,-2,1)\ar[ld,blue]\ar[rd,mygreen] \& v_{\overline{7}}(0,3,-2)\ar[d,red] \& \\
\& v_{\overline{13}}(-1,-1,1)\ar[rd,mygreen] \& \& v_{\overline{5}}(-1,1,-1)\ar[ld,blue]\ar[d,red,"{[2]}",pos=0.3] \& \\
\& \& v_{\overline{12}}(-1,2,-1)\ar[d,red] \& v_{\overline{4}}(-2,-1,0)\ar[ld,blue]\ar[d,red,"{[3]}",pos=0.3] \& \\
\& \& v_{\overline{11}}(-2,0,0)\ar[d,red,"{[2]}",pos=0.3] \& v_{\overline{3}}(-3,-3,1)\ar[ld,blue]\ar[rd,mygreen] \& \\
\& \& v_{\overline{10}}(-3,-2,1)\ar[rd,mygreen] \& \& v_{\overline{2}}(-3,0,-1)\ar[ld,blue] \\
\& \& \& v_{\overline{9}}(-3,1,-1)\ar[d,red] \& \\
\& \& \& v_{\overline{8}}(-4,-1,0)\ar[ld,blue] \& \\
\& \& v_{\overline{1}}(-4,0,0) \& \&
}
\caption{The 31-dimensional quantized adjoint module for $U_q$G$_{2\vert 1}$}
    \label{fig:G3 e-f action}
\end{figure}

\begin{prp}
    The basis $\{v_i\}_{i=1}^{31}$ exists.
\end{prp}
\begin{proof}
    We explicitly check all the relations of $U_q$G$_{2\vert 1}$. 
\end{proof}

After a $q\to 1$ limit, the module $L_{(4,0,0)}$ is the adjoint module for G$_{2\vert 1}$ and reflects the structure of the Lie superalgebra G$_{2\vert 1}$ itself. The even part of G$_{2\vert 1}$ splits as $\mathfrak{sl}_2\oplus $G$_2$. In our chosen basis, up to a $q\to 1$ limit, the vectors $v_1$, $v_{15}$, $v_{\overline{1}}$ form a basis of the subalgebra $\mathfrak{sl}_2$ of the even part, while the vectors $v_2,\dots,v_{7}$, $v_{16},v_{17}$, $v_{\overline{7}},\dots,v_{\overline{2}}$ form a basis of the subalgebra G$_2$ of the even part. The vectors $v_{8},\dots,v_{14}$ form the first fundamental representation of G$_2$, and so do the vectors $v_{\overline{14}},\dots, v_{\overline{8}}$. Together, they form a representation of $\mathfrak{sl}_2\oplus$G$_2$, namely the $2$-dimensional vector representation of $\mathfrak{sl}_2$ tensor the first fundamental representation of G$_2$. This tensor product is the odd part of the superalgebra G$_{2\vert 1}$.

\subsection{Quantum affine algebra and \texorpdfstring{$q$}--characters}

We consider the distinguished Dynkin diagram where one node is ferminonic and the affine node attaches to this fermionic node, as shown below:

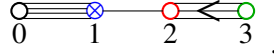
\begin{center}
\begin{tikzpicture}
    [root/.style={circle, draw, thick, minimum size=6pt, inner sep=0pt}]
    \node[root] at (0,0) {};
    \node at (1,0) {$\textcolor{blue}{\otimes}$};
    \node[root, color=red] at (2,0) {};
    \node[root, color=mygreen] at (3,0) {};

    \node at (2.5,0) {\scalebox{1.5}{$<$}};
    \draw (2.9,0)--(2.1,0);
    \draw (2.95,0.1)--(2.05,0.1);
    \draw (2.95,-0.1)--(2.05,-0.1);
    \draw (1.1,0) -- (1.9,0);
    \draw (0.1,0.03) -- (0.9,0.03);
    \draw (0.05,0.1) -- (0.95,0.1);
    \draw (0.1,-0.03) -- (0.9,-0.03);
    \draw (0.05,-0.1) -- (0.95,-0.1);

    \node at (0,-0.3) {0};
    \node at (1,-0.3) {1};
    \node at (2,-0.3) {2};
    \node at (3,-0.3) {3};
\end{tikzpicture}\ .
\end{center}
The corresponding Cartan matrix is given by 
$$\begin{bmatrix}
    2 & -1 & 0 & 0 \\
    -4 & 0 & 1 & 0 \\
    0 & -1 & 2 & -3 \\
    0 & 0 & -1 & 2
\end{bmatrix}\ .$$
We choose $d_0=-4$, $d_1=-1$, $d_2=1$, $d_3=3$ in this case.

The affine roots,  see \eqref{affine roots}, are given by 
\bee{
A_{1,a}=\mr{2}_a^{-1}\ ,\quad A_{2,a}=\mr{1}_{q^{-1}a}^{qa}\mr{2}_{qa}\mr{2}_{q^{-1}a}\mr{3}_a^{-1}\ ,\quad A_{3,a}=\mr{3}_{q^3a}\mr{3}_{q^{-3}a}\mr{2}_{q^2a}^{-1}\mr{2}_a^{-1}\mr{2}_{q^{-2}a}^{-1}\ .
}
Here $\mr{1}_b^a$, $\mr{2}_a$, $\mr{3}_a$, $a,b\in\C^\times$, are triples of rational functions given by
\bee{
\mr{1}_b^a=\bigg(\sqrt{a/b}\frac{1-z/a}{1-z/b},1,1\bigg)\ ,\quad \mr{2}_a=\bigg(1,q^{-1}\frac{1-qz/a}{1-q^{-1}z/a},1\bigg)\ ,\quad \mr{3}_a=\bigg(1,1,q^{-3}\frac{1-q^{3}z/a}{1-q^{-3}z/a}\bigg)\ .
}

We construct a $32$-dimensional $U_q\tl{\text{G}}_{2\vert 1}$-module $\tl{L}_{(4,0,0)}$.

The $q$-character of $\tl{L}_{(4,0,0)}$ is shown in Figure \ref{fig:G3 affine e-f action} below. Here we denote $\mr{1}_{q^b}^{q^a}$ by $\mr{1}_b^a$ and $\mr{i}_{q^a}$ by $\mr{i}_a$ for $i=2,3$. We omit writing the shift of the affine roots on the corresponding arrows, but it is clear from the corresponding $\ell$-weights. The supplement for this module is given as follows. The constants of the action of $X_i^-(z)$ are shown in Figure \ref{fig:G3 affine e-f action}. Here all the numbers on arrows are quantum numbers, and if a number is not shown on an arrow then it is one. As before, we have $\ol{i}=32-i$. The constants of action of $X_i^+(z)$ can be obtained from those of $X_i^-(z)$ shown in the diagram in Figure \ref{fig:G3 affine e-f action}. Specifically,  when $1\le j,k\le 14$, we have $X_i^+(z)\,\tl v_j=a\,\delta(z/q^s)\,\tl v_k$ if and only if $X_i^-(z)\,\tl v_{\overline{k}}=a\,\delta(z/q^s)\,\tl v_{\overline{j}}$, and when $18\le j,k\le 31$ we have $X_i^+(z)\,\tl v_j=(-1)^{s_i}a\,\D(z/q^s)\,\tl v_k$ if and only if $X_i^-(z)\,\tl v_{\overline{k}}=a\,\D(z/q^s)\,\tl v_{\overline{j}}$. In the case when $j=15,16,17,32$ and $k=6,7,14$, we have $X_i^+(z)\,\tl v_j=a\,\D(z/q^s)\,\tl v_k$ if and only if $X_i^-(z)\,\tl v_j=a\,\D(z/q^s)\,\tl v_{\overline{k}}$, and $X_i^+(z)\,\tl v_{\ol{k}}=a\,\D(z/q^s)\,\tl v_{j}$ if and only if $X_i^-(z)\,\tl v_k=a\,\D(z/q^s)\,\tl v_j$.
\begin{figure}[ht!]
    \centering
\ddia{row sep=10pt, column sep=-1pt}{
\& \& \tl{v}_1\big(\mr{1}_0^8\big) \ar[ld,blue,"{[4]}"] \& \& \\
\& \tl{v}_8\big(\mr{1}_0^8\, \mr{2}_0\big) \ar[d,red] \& \& \& \\
\& \tl{v}_9\big(\mr{1}_2^8\,\mr{2}_2^{-1}\mr{3}_1\big) \ar[rd,mygreen]\ar[ld,blue,"{[3]}"] \& \& \& \\
\tl{v}_2\big(\mr{1}_2^8\,\mr{3}_1\big) \ar[rd,mygreen] \& \& \tl{v}_{10}\big(\mr{1}_2^8\,\mr{2}_4\mr{2}_6\mr{3}_7^{-1}\big) \ar[d,red]\ar[ld,blue,"{[3]}"] \& \& \\
\& \tl{v}_3\big(\mr{1}_2^8\, \mr{2}_2\mr{2}_4\mr{2}_6\mr{3}_7^{-1}\big) \ar[d,red] \& \tl{v}_{11}\big(\mr{1}_2^6\,\mr{2}_4\mr{2}_8^{-1}\big) \ar[d,red,"{[2]}",pos=0.25]\ar[ld,blue,"{[2]}"] \& \& \\
\& \tl{v}_4\big(\mr{1}_2^6\,\mr{2}_2\mr{2}_4\mr{2}_8^{-1}\big) \ar[d,red,"{[2]}",pos=0.25] \& \tl{v}_{12}\big(\mr{1}_2^4\,\mr{2}_6^{-1}\mr{2}_8^{-1}\mr{3}_5\big) \ar[ld,blue]\ar[rd,mygreen] \& \& \\
\& \tl{v}_5\big(\mr{1}_2^4\,\mr{2}_2\mr{2}_6^{-1}\mr{2}_8^{-1}\mr{3}_5\big) \ar[d,red,"{[3]}",pos=0.25]\ar[rd,mygreen] \& \& \tl{v}_{13}\big(\mr{1}_2^4\,\mr{2_{10}\mr{3}_{11}^{-1}}\big) \ar[ld,blue]\ar[d,red] \& \\
\& \tl{v}_7\big(\mr{2}_4^{-1}\mr{2}_6^{-1}\mr{2}_8^{-1}\mr{3}_3\mr{3}_5\big) \ar[rrd,mygreen,"{[4]}"',pos=0.15]\ar[rrrd,mygreen] \& \tl{v}_6\big(\mr{1}_2^4\,\mr{2}_2\mr{2}_{10}\mr{3}_{11}^{-1}\big) \ar[d,red,"{[5]}",pos=0.1]\ar[rd,red,"{[3]}",pos=0.15] \& \tl{v}_{14}\big(\mr{1}_{2,12}^{4,10}\,\mr{2}_{12}^{-1}\big) \ar[lld,blue,"{-[4]}",pos=0.95] \ar[ld,blue,"{[4]}",pos=0.1] \&
\\ [0.5cm]
\& \tl{v}_{15}\big(\mr{1}_{2,12}^{4,10}\big) \ar[d, blue,"{1/[5]}",pos=0.25] \& \tl{v}_{16}\big(\mr{1}_{2,12}^{4,10}\,\mr{2}_2\mr{2}_{12}^{-1}\big) \ar[ld,blue,"{1/[5]}",pos=0.8] \ar[d,red,"{1/[4]}",pos=0.25] \& \tl{v}_{17}\big(\mr{2}_{10}\mr{2}_4^{-1}\mr{3}_3\mr{3}_{11}^{-1}\big) \ar[d,mygreen]\ar[ld,red,"{1/[4]}",pos=0.3] \& \tl{v}_{32}\big(\mr{3}_5\mr{3}_9^{-1}\big) \ar[ld,mygreen,"{-[2]}",pos=0.3] \\
\& \tl{v}_{\ol{14}}\big(\mr{1}_{2,12}^{4,10}\,\mr{2}_2\big) \ar[d,red] \& \tl{v}_{\ol{6}}\big(\mr{1}_{12}^{10}\,\mr{2}_4^{-1}\mr{2}_{12}^{-1}\mr{3}_3\big) \ar[ld,blue]\ar[rd,mygreen] \& \tl{v}_{\ol{7}}\big(\mr{2}_6\mr{2}_8\mr{2}_{10}\mr{3}_9^{-1}\mr{3}_{11}^{-1}\big) \ar[d,red] \& \\
\& \tl{v}_{\ol{13}}\big(\mr{1}_{12}^{10}\,\mr{2}_4^{-1}\mr{3}_3\big) \ar[rd,mygreen] \& \& \tl{v}_{\ol{5}}\big(\mr{1}_{12}^{10}\,\mr{2}_6\mr{2}_8\mr{2}_{12}^{-1}\mr{3}_9^{-1}\big) \ar[ld,blue]\ar[d,red,"{[2]}",pos=0.25] \& \\
\& \& \tl{v}_{\ol{12}}\big(\mr{1}_{12}^{10}\,\mr{2}_6\mr{2}_8\mr{3}_9^{-1}\big) \ar[d,red] \& \tl{v}_{\ol{4}}\big(\mr{1}_{12}^8\,\mr{2}_6\mr{2}_{10}^{-1}\mr{2}_{12}^{-1}\big) \ar[ld,blue]\ar[d,red,"{[3]}",pos=0.25] \& \\
\& \& \tl{v}_{\ol{11}}\big(\mr{1}_{12}^8\,\mr{2}_6\mr{2}_{10}^{-1}\big) \ar[d,red,"{[2]}",pos=0.25] \& \tl{v}_{\ol{3}}\big(\mr{1}_{12}^6\,\mr{2}_8^{-1}\mr{2}_{10}^{-1}\mr{2}_{12}^{-1}\mr{3}_7\big) \ar[ld,blue]\ar[rd,mygreen] \& \\
\& \& \tl{v}_{\ol{10}}\big(\mr{1}_{12}^6\,\mr{2}_8^{-1}\mr{2}_{10}^{-1}\mr{3}_7\big)\ar[rd,mygreen] \& \& \tl{v}_{\ol{2}}\big(\mr{1}_{12}^6\,\mr{3}_{13}^{-1}\big) \ar[ld,blue] \\
\& \& \& \tl{v}_{\ol{9}}\big(\mr{1}_{12}^6\,\mr{2}_{12}\mr{3}_{13}^{-1}\big) \ar[d,red] \& \\
\& \& \& \tl{v}_{\ol{8}}\big(\mr{1}_{14}^6\,\mr{2}_{14}^{-1}\big) \ar[ld,blue] \& \\
\& \& \tl{v}_{\ol{1}}\big(\mr{1}_{14}^6\big) \& \&
}
\caption{The $q$-character and the supplement of the smallest non-trivial module for $U_q\tl{\text{G}}_{2\vert 1}$}
    \label{fig:G3 affine e-f action}
\end{figure}

The restriction of $\tl L_{(4,0,0)}(z)$ to $U_q$G$_{2\vert 1}$ decomposes as 
$$\tl{L}_{(4,0,0)}(a)\cong L_{(4,0,0)}\oplus L_{(0,0,0)}\ .$$

We now describe the action of $U_q\tl{\text{G}}_{2\vert 1}$ on $\tl{L}_{(4,0,0)}(z)$ in the Drinfeld-Jimbo realization. We choose a basis $\{v_i\}_{i=1}^{31}\cup\{v_{32}\}$ for $\tl{L}_{(4,0,0)}(z)$,
 where $\{v_i\}_{i=1}^{31}$ is the basis of $U_q$G$_{2\vert 1}$-module $L_{(4,0,0)}$ as in Figure \ref{fig:G3 e-f action} and $v_{32}$ is a basis of the $1$-dimensional $U_q$G$_{2\vert 1}$-module $L_{(0,0,0)}$. Note that $v_i$ are scalar multiples of $\tl{v}_i$ for all $i$, except for $i=15,16,17,32$.

Let
\eq{\label{superG:E0 action}
E_0=z\,\bigg( -\frac{[2]_6}{[3]^{\mr{i}}}E_{15,1} + [4]\,E_{16,1} - \frac{[2]_2[3]}{[3]^{\mr{i}}} E_{17,1} + \frac{[3]}{[3]^{\mr{i}}} E_{32,1} + E_{\overline{1},15} + \big([2]^{\mr{i}}\big)^2 E_{\overline{1},32} +\sum_{i=8}^{14}E_{10+i,i}\bigg)\ ,
}
\eq{\label{superG: F0 action}
F_0=z^{-1}\,\bigg( \frac{[2]_6}{[3]^{\mr{i}}} E_{15,\overline{1}} - [4]\, E_{16,\overline{1}} + \frac{[2]_2[3]}{[3]^{\mr{i}}} E_{17,\overline{1}} - \frac{[3]}{[3]^{\mr{i}}} E_{32,\overline{1}} - E_{1,15} -\big([2]^{\mr{i}}\big)^2 E_{1,32} + \sum_{i=8}^{14}E_{i,10+i} \bigg)\ .
}
where $E_{ij}$ are matrix units, that is, $E_{ij}v_k=\D_{jk}v_i$.

\begin{prp}
    The action of $E_i, F_i$, $i=1,2,3,$ as in Figure  \ref{fig:G3 e-f action} and $E_0, F_0$ as in \eqref{superG:E0 action}, \eqref{superG: F0 action}  gives a $U_q\tl{\text{G}}_{2\vert 1}$-module structure.
\end{prp}
\begin{proof}
    We check all the relations in $U_q\tl{\text{G}}_{2\vert 1}$ by a direct computation.
\end{proof}

\subsection{The \texorpdfstring{$R$}--Matrix}
We start studying the tensor square of $\tl{L}_{(4,0,0)}$. This module is completely reducible.
\begin{thm}
As $U_q$G$_{2\vert 1}$-modules, we have 
\eq{\label{superG: tensor adjoint}
\big(\tl{L}_{(4,0,0)}(z)\big)^{\otimes 2}\cong\big(\underbracket[0.1ex]{L_{(4,0,0)}}_{31}\oplus \underbracket[0.1ex]{L_{(0,0,0)}}_{1}\big)^{\otimes 2}\cong \underbracket[0.1ex]{L_{(8,0,0)}}_{192}\oplus \underbracket[0.1ex]{L_{(8,1,0)}}_{448}\oplus \underbracket[0.1ex]{L_{(7,0,1)}}_{289}\oplus\, 3\underbracket[0.1ex]{L_{(4,0,0)}}_{31}\oplus\, 2\underbracket[0.1ex]{L_{(0,0,0)}}_{1}\ .
}
\end{thm}
\p{We make explicit computations in $\tl{L}_{(4,0,0)}^{\otimes 2}$ to  observe the following.
There are $8$ linearly independent singular vectors, which we enumerate as $u_1,u_2,u_3,u_{4a},u_{4b},u_{4c},u_{5a},u_{5b}$. The singular vector $u_1=v_1\otimes v_1$ generates $L_{(8,0,0)}$, the singular vector $u_2=q^{-2}v_1\otimes v_{8}-q^2 v_{8}\otimes v_1$ generates $L_{(8,1,0)}$, and the singular vector $u_3=q^{-3}v_1\otimes v_2+q^3 v_2\otimes v_1 - q^{1/2} v_{8}\otimes v_{9}+q^{-1/2}v_{9}\otimes v_{8}$ generates $L_{(7,0,1)}$. There is a natural choice for the three singular vectors $u_{4a}$, $u_{4b}$, $u_{4c}$ of weight $(4,0,0)$. We choose $u_{4b}=v_1\otimes v_{32}$, $u_{4c}=v_{32}\otimes v_1$, and $u_{4a}\in L_{(4,0,0)}^{\otimes 2}$ is given by 
\bee{
u_{4a} =\ & -[3]_2^{\mr{i}}\,(q^3v_1\otimes v_{15}-q^{-3}v_{15}\otimes v_1)+[2]_3\,(q^3v_1\otimes v_{16}-q^{-3}v_{16}\otimes v_1)-[3]\,(q^3v_1\otimes v_{17}-q^{-3}v_{17}\otimes v_1) \\
& -[2]_3\,(q^5v_8\otimes v_{14}+q^{-5}v_{14}\otimes v_8-q^4v_9\otimes v_{13}-q^{-4}v_{13}\otimes v_9+qv_{10}\otimes v_{12}+q^{-1}v_{12}\otimes v_{10}) + [3]^{\mr{i}} v_{11}\otimes v_{11}\ . 
}
Finally, the two singular vectors of weight zero are chosen so that $u_{5b}=v_{32}\otimes v_{32}$, and $u_{5a}\in L_{(4,0,0)}^{\otimes 2}$ is given by
\bee{
u_{5a}=\ & [3]^{\mr{i}}\,(q^{10}v_1\otimes v_{\ol{1}}+q^{-10}v_{\ol{1}}\otimes v_1) + [3]_2[4]\,(q^9v_2\otimes v_{\ol{2}}+q^{-9}v_{\ol{2}}\otimes v_2-q^6 v_3\otimes v_{\ol{3}}-q^{-6}v_{\ol{
3
}}\otimes v_3 \\
& + q^3 v_7\otimes v_{\ol{7}} + q^{-3}v_{\ol{7}}\otimes v_7) + [4][3]^{\mr{i}}\,(q^5v_4\otimes v_{\ol{4}}+q^{-5}v_{\ol{4}}\otimes v_4-q^4v_5\otimes v_{\ol{5}}-q^{-4}v_{\ol{5}}\otimes v_5+q v_6\otimes v_{\ol{6}} \\ 
& + q^{-1}v_{\ol{6}}\otimes v_6 + q^{10}v_8\otimes v_{\ol{8}}-q^{-10}v_{\ol{8}}\otimes v_8-q^9 v_9\otimes v_{\ol{9}} + q^{-9} v_{\ol{9}}\otimes v_9 + q^6 v_{10}\otimes v_{\ol{10}}-q^{-6}v_{\ol{10}}\otimes v_{10} \\ 
& + q^4 v_{12}\otimes v_{\ol{12}} -q^{-4}v_{\ol{12}}\otimes v_{12}-q v_{13}\otimes v_{\ol{13}}+q^{-1} v_{\ol{13}}\otimes v_{13}+v_{14}\otimes v_{\ol{14}}-v_{\ol{14}}\otimes v_{14}-v_{15}\otimes v_{16}-v_{16}\otimes v_{15}) \\
& - [2]_2[3]^{\mr{i}}\,(q^5 v_{11}\otimes v_{\ol{11}}-q^{-5} v_{\ol{11}}\otimes v_{11}) + [2]_2[3]_2^{\mr{i}}\,v_{15}\otimes v_{15} + [2]_2[3]\,(v_{15}\otimes v_{17}+v_{17}\otimes v_{15}-v_{17}\otimes v_{17})\ .
}
}

For $\la=(8,0,0),(8,1,0),(7,0,1),(4,0,0),(0,0,0)$, let $P_\la^q$ be the projector onto the summand $L_{\la}$ in the decomposition \eqref{superG: tensor adjoint}.

\begin{thm}
\label{superG:Rcheck}
In terms of projectors, we have 

\eq{\label{superG: Rqz}
\check{R}(z)=\ & P^q_{(8,0,0)}-q^8\frac{1-q^{-8}z}{1-q^8 z}\,P^q_{(8,1,0)}+q^{14}\frac{(1-q^{-6}z)(1-q^{-8}z)}{(1-q^6 z)(1-q^8 z)}\,P^q_{(7,0,1)}
+ \frac{q^5\ f_1(z)}{(1-q^{-4}z)(1-q^6 z)(1-q^8 z)}\otimes P^q_{(4,0,0)} \\
& + \frac{q^2\ f_0(z)}{(1-q^{-4}z)(1-q^{-6}z)(1-q^6 z)(1-q^8 z)}\otimes P^q_{(0,0,0)}\ ,
}
where 
$$f_1(z)=\begin{bmatrix}
    -q^{-3}+\A_q\,z-\A_{q^{-1}}\,z^2+q^3z^3 & \hspace{-20pt} \B\,z(1-z) & \hspace{-10pt} \B\,z(1-z) 
    \vspace{5pt} \\
    \G\,z(1-z) & \hspace{-20pt} \mu\,z(q^3+q^{-3}z) & \hspace{-10pt} (1-z)(q^3-\nu\,z+q^{-3}z^2) 
    \vspace{5pt} \\
    \G\,z(1-z) & \hspace{-20pt} (1-z)(q^3-\nu\,z+q^{-3}z^2) & \hspace{-10pt} \mu\,z(q^3+q^{-3}z)
\end{bmatrix} \ ,$$

$$ f_0(z)= \begin{bmatrix}
    q^{-6}-q^{-3}\zeta\,z+\xi\,z^2-q^3\zeta\,z^3+q^6z^4 & \eta\, z(1-z^2) 
    \vspace{5pt} \\
    \rho\, z(1-z^2) & q^{6}-q^{3}\zeta\,z+\xi\,z^2-q^{-3}\zeta\,z^3+q^{-6}z^4
\end{bmatrix} \ .
$$

\bigskip

Here, the constants $\A_q,\B,\G,\mu,\nu,\zeta,\xi,\eta,\rho\in\C(q)$ are given by 
\bee{
& \A_q=\frac{(q^5-q+2q^{-7}-q^{-9})\,[3]}{[3]^{\mr{i}}}\ ,\quad \B=\frac{([2]^{\mr{i}})^3\,[2]_2\,[4]}{[3]^{\mr{i}}}\ ,\quad \G=\frac{[2]^{\mr{i}}\,[3]_2^{\mr{i}}\,[2]\,([3])^2
}{[3]^{\mr{i}}}\ ,
\vspace{5pt} \\
& \mu=\frac{([2]^{\mr{i}})^3\,[3]\,[4]}{[3]^{\mr{i}}}\ ,\quad \nu=\frac{[2]_7+[2]_3-[2]}{[3]^{\mr{i}}}\ ,\quad \zeta= \frac{[4]\,[3]_2^{\mr{i}}}{[3]^{\mr{i}}}\ ,
\vspace{5pt} \\ 
& \xi=\frac{[2]_2\,[3]_2\,([2]_8-[2]_6-[2]_4+[3])}{[3]^{\mr{i}}}\ ,\quad \eta=\frac{([2]^{\mr{i}})^5\,[4]}{[3]^{\mr{i}}}\ ,\quad \rho=\frac{[2]^{\mr{i}}\,([3])^2\,[3]_3\,[4]}{[3]^{\mr{i}}}\ . 
}
\p{
We use the equation $$\check{R}(z)\,\Delta(F_0)=\Delta(F_0)\,\check{R}(z):\tl L_{(4,0,0)}(z)\otimes\tl{L}_{(4,0,0)}(1)\to \tl L_{(4,0,0)}(1)\otimes\tl{L}_{(4,0,0)}(z)\ ,$$ with $F_0$ in \eqref{superG: F0 action}. Equivalently, one can use $E_0$ in \eqref{superG:E0 action}.
}
\end{thm}

\begin{rmk}
    As in the case of $U_q\tl{\text{F}}_{3\vert 1}$, here also, in some sense, there are both zeros and poles at the same location, $z=q^{-6}$. Further, the determinants of 
    $$g_1(z)=\frac{q^5\ f_1(z)}{(1-q^{-4}z)(1-q^6 z)(1-q^8 z)}\ ,\quad g_2(z)=\frac{f_0(z)}{(1-q^{-4}z)(1-q^{-6}z)(1-q^6 z)(1-q^8 z)}\ ,$$
     are given by 
     $$\det\big(g_1(z)\big)=q^{18}\frac{(1-q^4z)(1-q^{-6}z)(1-q^{-8}z)^2}{(1-q^{-4}z)(1-q^6z)(1-q^8z)^2}\ ,\quad \det\big(g_2(z)\big)=q^4\frac{(1-q^4z)(1-q^{-8}z)}{(1-q^{-4}z)(1-q^8z)}\ .$$
     Similar to $U_q$F$_{3\vert 1}$, we have poles at $z=q^{-6}$ and $z=q^6$ simultaneously.
\end{rmk}

The tensor product $\tl{L}_{(4,0,0)}(z)\otimes\tl{L}_{(4,0,0)}(1)$ is irreducible except for some special values of $z$. These special values of $z$, and the corresponding non-trivial submodules, are listed below. Note that the submodules for $z=q^k$, are quotient modules for $z=q^{-k}$.
\begin{enumerate}
    \item $z=q^4$: We have a $992$-dimensional submodule isomorphic to $L_{(8,0,0)}\oplus L_{(8,1,0)}\oplus L_{(7,0,1)}\oplus  2\,L_{(4,0,0)}\oplus L_{(0,0,0)}$ as a $U_q$G$_{2\vert 1}$-module, and generated by $u_1$.
    \item $z=q^{-4}$: We have a $32$-dimensional submodule isomorphic to $L_{(4,0,0)}\oplus L_{(0,0,0)}$ as a $U_q$G$_{2\vert 1}$-module, and generated by $u_{4a}-u_{4b}-u_{4c}$.
    \item $z=q^{6}$: We have a $321$-dimensional submodule isomorphic to $ L_{(7,0,1)}\oplus L_{(4,0,0)}\oplus L_{(0,0,0)}$ as a $U_q$G$_{2\vert 1}$-module, and generated by $u_3$. This submodule is contained inside a $1023$-dimensional submodule isomorphic to $L_{(8,0,0)}\oplus L_{(8,1,0)}\oplus L_{(7,0,1)}\oplus 3\,L_{(4,0,0)}\oplus L_{(0,0,0)}$ as a $U_q$G$_{2\vert 1}$-module, and generated by $u_1$.
    \item $z=q^{-6}$: We have a $1$-dimensional submodule isomorphic to $L_{(0,0,0)}$ as a $U_q$G$_{2\vert 1}$-module, and generated by $u_{5a}-u_{5b}$. This submodule is contained inside a $703$-dimensional submodule generated by $u_1$, and isomorphic to $L_{(8,0,0)}\oplus L_{(8,1,0)} 2\,L_{(4,0,0)}\oplus L_{(0,0,0)}$ as a $U_q$G$_{2\vert 1}$-module
    \item $z=q^8$: We have a $800$-dimensional submodule isomorphic to $L_{(8,1,0)}\oplus L_{(7,0,1)}\oplus  2\,L_{(4,0,0)}\oplus L_{(0,0,0)}$ as a $U_q$G$_{2\vert 1}$-module, and generated by $u_2$.
    \item $z=q^{-8}$: We have a $224$-dimensional submodule isomorphic to $L_{(8,0,0)}\oplus L_{(4,0,0)}\oplus L_{(0,0,0)}$ as a $U_q$G$_{2\vert 1}$-module, and generated by $u_1$.
\end{enumerate}

We consider the $q\to 1$ rational limit. With our normalization of the singular vectors $u_1,\dots, u_{5b}$, the limit of these vectors exists, and therefore the limit of the projectors exists. Let $P_{(a,b,c)}$ be the limit $q\to 1$ of the $U_q$G$_{2\vert 1}$-projectors $P_{(a,b,c)}^q$ appearing in the expression \eqref{superG: Rqz} of $\check{R}(z)$.

\begin{cor}
    The rational $R$-matrix $\check{R}(u)$ for the Yangian of G$_{2\vert 1}$ is given by 
    \eq{\label{superG: rational R matrix}
\check{R}(u) =\ & P_{(8,0,0)} + \frac{4-u}{4+u}\,P_{(8,1,0)} + \frac{(3-u)(4-u)}{(3+u)(4+u)}\,P_{(7,0,1)} + \frac{\bar f_1(u)}{(2-u)(3+u)(4+u)}\otimes P_{(4,0,0)} \\
& + \frac{\bar f_0(u)}{(2-u)(3-u)(3+u)(4+u)}\otimes P_{(0,0,0)}\ ,
}
where the matrices $\bar f_1(u)$ and $\bar f_0(u)$ are given by 
\bee{
& \bar f_1(u)=\begin{bmatrix}
    24-10u+3u^2+u^3 &  -8\,u &  -8\,u 
    \vspace{5pt} \\
    -18\,u & 24 & u(5-u)(2+u) 
    \vspace{5pt} \\
    -18\,u & u(5-u)(2+u) & 24
\end{bmatrix} \ ,
\vspace{10pt} \\
\vspace{10pt} \\
& \bar f_0(u)= \begin{bmatrix}
    72-30u-u^2+6u^3+u^4 & -8\,u 
    \vspace{5pt} \\
    -216\,u & 72+30u-u^2-6u^3+u^4
\end{bmatrix} \ .
}
\p{
We substitute $z=q^{2u}$ in \eqref{superG: Rqz} and take limit $q\to 1$. Apriori, the $q\to 1$ limit of $\check{R}(q^{2u})$ does not exist. To have a limit we conjugate $\check{R}(q^{2u})$ by $A\otimes A$ where the operator $A:V\to V$ is given by $A\,v_i=v_i$ for $1\le i\le 31$ and $A\,v_{32}=\frac{1}{[2]^{\mr{i}}}v_{32}$. Thus, we obtain $\check{R}(u)=\lim_{q\to 1}(A\otimes A)\check{R}(q^{2u})(A\otimes A)^{-1}$ as above in \eqref{superG: rational R matrix}. Note that conjugation by an operator of the form $A\otimes A$ preserves the quantum Yang-Baxter equation and, for our choice of  $A$, $A\otimes A$  commutes with action of  $U_q{\text{G}}_{2|1}$ but not with action of $U_q\tl{\text{G}}_{2|1}$.
}
\end{cor}

{\bf Acknowledgments.\ }
The authors are partially supported by Simons Foundation grant number \#709444.

\end{document}